\documentclass[10pt,a4paper,reqno,oneside]{amsart}

\usepackage{amsrefs}

\usepackage{array}
\usepackage{url}
\usepackage{multicol}
\usepackage[linktocpage = true]{hyperref}

\hypersetup{
 colorlinks = true,
 citecolor = blue,
}
\usepackage{color}
\definecolor{red}{rgb}{1,0,0}
\definecolor{green}{rgb}{0,1,0}
\definecolor{blue}{rgb}{0,0,1}
\definecolor{refkey}{gray}{.625}
\definecolor{labelkey}{gray}{.625}
\usepackage[a4paper]{geometry}
\usepackage{amsmath,amssymb}
\usepackage{amssymb}
\usepackage{amsmath}
\usepackage{amsthm} 
\usepackage{enumerate} 
\usepackage{graphicx} 
\usepackage{todonotes}
\usepackage{tikz-cd}
\allowdisplaybreaks

\makeatletter

\theoremstyle{plain}
\newtheorem{thm}{\protect\theoremname}[section]
\newtheorem{prop}[thm]{\protect\propositionname}
\newtheorem{cor}[thm]{\protect\corollaryname}
\newtheorem{lem}[thm]{\protect\lemmaname}

\theoremstyle{definition}
\newtheorem{example}[thm]{\protect\examplename}
\newtheorem{defn}[thm]{\protect\definitionname}

\newtheorem{notation}[thm]{\protect\notationname}

\AtBeginDocument{
	
}

\makeatother

  \providecommand{\corollaryname}{Corollary}
  \providecommand{\examplename}{Example}
  \providecommand{\lemmaname}{Lemma}
  \providecommand{\propositionname}{Proposition}
  \providecommand{\theoremname}{Theorem}
  \providecommand{\definitionname}{Definition}
  \providecommand{\remarkname}{Remark}
  \providecommand{\notationname}{Notation}

  \makeatletter
  \newcommand{\be}{%
  \begingroup
  \eqnarray%
   \@ifstar{\nonumber}{}%
  }

\newcommand{\xx}{x}
\newcommand{\A}{\mathcal{A}}
\newcommand{\xA}{\mathbf{A}}

\newcommand{\ev}{\operatorname{ev}}

\newcommand{\T}{\mathbb{T}}
\newcommand{\e}{\mathbf{e}}

\newcommand{\Gal}{\operatorname{Gal}}

\newcommand{\n}{\mathfrak{n}}
\newcommand{\p}{\mathfrak{p}}
\newcommand{\f}{\mathfrak{f}}
\newcommand{\h}{\mathbf{h}}
\newcommand{\HH}{\mathbf{H}}
\newcommand{\m}{\mathfrak{m}}
\newcommand{\M}{\mathcal{M}}

\newcommand{\Mod}{\operatorname{Mod}}

\newcommand{\D}{\operatorname{D}}
\newcommand{\Weil}{\operatorname{Weil}}

\newcommand{\oper}{\operatorname{O}}

\newcommand{\Res}{\operatorname{Res}}

\newcommand{\sep}{\mathrm{sep}}

\newcommand{\abs}[1]{\left|  #1 \right|}

\makeatother

\begin{document}

\title{Relation between Anderson Generating Functions and Weil Pairing}
%\thanks{ Research partially supported by NSFC grants
%	12071247/12101616(Hu),  Guangdong Basic and Applied Basic Research Foundation No. 2021A1515110654(Huang) and
% the Basic and Applied Basic Research of Guangzhou Basic Research Program No.
%202201010234(Huang) .}
%\thanks{$^*$ The corresponding author.}
%
\author{Chuangqiang Hu}
\address{Sun Yat-Sen University, School of Mathematics, Guangzhou, China}
\email{\href{huchq@mail2.sysu.edu.cn}{huchq@mail2.sysu.edu.cn}}

\author{Yixuan Ou-Yang}
\address{Sun Yat-Sen University, School of Mathematics, Guangzhou, China}
\email{\href{ouyyx5@mail2.sysu.edu.cn}{ouyyx5@mail2.sysu.edu.cn}}

\allowdisplaybreaks

\begin{abstract}
The existence of the Weil pairing for Drinfeld modules was proved by van~der~Heiden using the Anderson $t$-motive. Papikian's note provided the explicit formula for the rank-two Weil pairing  that avoids Anderson motives. Following this approach, Katen extended the formula to higher ranks. As Papikian observed, this method is more elementary than the approach using Anderson motives, but it is less conceptual. This paper is devoted to a new insight into Katen's formula motivated by the Moore determinant coming from Hamahata's tensor product of Drinfeld modules and the basis of torsion modules found by Maurischat and Perkins. We investigate the Weil operator, establish its connection with the remainder polynomial of Anderson generating functions modulo a fixed polynomial $\f$, and finally derive an extremely simple interpretation: the value of the rank-$r$ Weil pairing is essentially the specific coefficient in the Moore determinant of certain Anderson generating functions.  
\end{abstract}
\maketitle{}


\tableofcontents

\section{Introduction}

\subsection{The Weil Pairing of Drinfeld Modules}
In classical theory, the Weil pairing for elliptic curves $ E $ is a perfect bilinear map from the $m$-torsion points to the $m$-th roots $ \mu_m $ of unity, formally written as 
\[
	E[m] \times E[m] \to \mu_m .	
\]
Drinfeld modules, introduced by V.~G.~Drinfeld \cite{DVG74}, are the function field analogues of elliptic curves. Let $ \xA = \mathbb{F}_q[\xx]$ be a polynomial ring and let $ K $ be an $ \xA $-field containing $ \mathbb{F}_q(\theta)$. 
Denote by $ K\{ \tau \} $ a twisted polynomial ring over $ K $.
We are interested in Drinfeld $ \xA$-modules of rank-$ r $ over $ K $, i.e., $\mathbb{F}_q$-algebra homomorphisms $ \xA \to K \{ \tau \} $ such that 
\begin{equation}\label{eq:Drinfelda}
    \phi: a(\xx) \mapsto \phi_a :=  a(\theta) + g_1 \tau^1 + \cdots + g_{r} \tau^r  
\end{equation}
for some coefficients $ g_1, \cdots, g_r \in K $. 
Given a Drinfeld module $ \phi $ and a polynomial $ \f \in \xA$ (or an ideal), the group $  \phi[\f] $ which carries the structure of a rank-$r$ $\xA$-module is the counterpart to the torsion points of an elliptic curve. In the function‑field setting there is a perfect multi-linear map whose image lies in $ \psi[\f] $ for some rank-one Drinfeld module $\psi$. In other words, the correct form of the Weil pairing in function field arithmetic is 
\[
	\Weil_\f :~ \prod_{i=1}^r  \phi [\f]  \to \psi [\f].
\]
\begin{defn}\label{defn:weil}[See \cite{P23}*{Section 3.7}]
Analogous to the classical theory,  the Weil pairing can be formally defined as a multilinear map $\Weil_\f $ satisfying the following properties:
\begin{enumerate}
	\item The map $\Weil_\f $ is $\xA$-multilinear, i.e. it is $\xA$-linear in each component.

	\item It is alternating: if $\mu_i = \mu_j$ for some $i \neq j$, then $\Weil_\f(\mu_1, \cdots ,\mu_r) = 0$.
	
	\item It is surjective and nondegenerate.
	\item It is Galois invariant:
	\[
	\sigma \Weil_\f(\mu_1, \cdots ,\mu_r) =\Weil_\f(\sigma\mu_1, \cdots ,\sigma\mu_r) ~{\rm for~ all} ~\sigma \in \Gal(  K^{\sep}/K).
	\]
	
	\item It satisfies the following compatibility condition for polynomials $ \m , \n \in \xA $ and $\mu_1, \cdots ,  \mu_r\in \psi [\m \n ]$:
	\[
	\psi_\n \Weil_{ \m \n }(\mu_1, \cdots ,\mu_r) =\Weil_\m (\phi_\n \mu_1, \cdots ,\phi_\n \mu_r). 
	\]
\end{enumerate}
\end{defn}
% \subsection{Weil pairing and the Galois actions}
% We provide a brief interpretation for the link of Weil pairing and the Galois actions. 
% The torsion $A$-module $\ker(\phi_I)$ is naturally equipped with an action of the absolute Galois group
% $ G_K = \Gal( K^\sep / K)$.
% Since the Galois action commutes with the $\phi$-action on $  \phi[I] $, we have the representation:
% \[
% 	\rho_\phi:  G_K \to \Aut_{A}(\ker (\phi_I)) \cong  \GL_r (A/I)
% \]
% Composing with the determinant map $ \det : \GL_r (A/I)\to (A/I)^{*}  $, we get a representation
% homomorphism
% \[
% 	  G_K \to(A/I)^{*}  
% \]
% which shall be viewed as 
% \[ 
% 	 \rho_\psi:  G_K \to  \psi[I] 
% \]
% for some rank-one Drinfeld module $\psi $. In this way, the Weil pairing shall be compatible with Galois actions. 
\subsection{Van~der~Heiden's Construction}
In \cite{vdHGJ04}, van~der~Heiden constructed the Weil pairing $ \Weil $ for general Drinfeld modules by extending the theory of Anderson $t$-motives \cite{AGW86}. 
The difficulty in defining the Weil pairing lies in  understanding the exterior product $ \psi $ in the category of Drinfeld modules.
More generally, it is also not obvious how to define tensor
products or take subquotients. However, van~der~Heiden observed that in the equivalent category of pure Anderson motives, the tensor product is closed under the operations of taking subquotients and tensor products. This construction coincides with Hamahata's notion \cite{HY93} in the case of polynomial rings.
In this way, van~der~Heiden showed that the rank-one Drinfeld module $ \psi  $ can be represented by 
\[ \psi_{\xx} =  (-1)^{r-1} g_r \tau + \theta , \]
for Drinfeld module $ \phi $ of the form \eqref{eq:Drinfelda}, thereby guaranteeing the existence of the Weil pairing. 
\subsection{Explicit Formula}
Section 3.7.2 of Papikian's note \cite{P23} discusses an alternative approach to the rank-two Weil pairing based on explicit formulas and remarks that 
\textit{this approach is much more elementary
than the approach via Anderson motives but has the disadvantage of being less
conceptual}. Inspired by this construction,   
 Katen \cite{KJ21} gave an explicit and elementary proof of the existence of the Weil pairing. 
In \cite{Hu24}*{Theorem 5.9}, the authors gave an interpretation of how van~der~Heiden's construction induces Katen's formula.

Let $ \M $ denote the Moore determinant and let $\diamond_{\phi}$ denote the Drinfeld action in $r$ variables (see Section \ref{sec:Katen}). Then Katen's formula for Weil pairings is precisely given by 
\begin{equation}\label{eq:weil0}
	\Weil_{\f} (\mu_1, \cdots, \mu_r ) = \M( \oper_\f^{(r)} \diamond_\phi (\mu_1\otimes \cdots\otimes\mu_r)), 	
\end{equation}
where $ \oper_\f^{(r)}$ is a symmetric polynomial in the variables $X_1,\cdots, X_r$, called the rank-$r$ Weil operator. 
According to Katen's definition, the rank-$r$ Weil operators $ \oper_\f^{(r)} $ for a modulus $ \f(\xx) = (\xx - \zeta) \m $ are derived from the following recursive formula 
\begin{align*}
    \oper_{\f}^{(r)}(X_1,\cdots, X_r) = & \oper_{\f}^{(r-1)}(X_1,\cdots, \hat{X_l}, \cdots, X_r) \m (X_l) \\
    & + \left(\prod_{j \neq l; j=1 }^r (X_j - \zeta )\right) \oper_{\m}^{(r)} (X_1, \cdots, X_r)
\end{align*}
for any index $l \in \{ 1,\cdots, r\}$. In particular, the rank-two operator (see \cite{P23}) is
\[
    \oper_\f^{(2)}(X_1, X_2) = \sum_{j=1}^{n} a_j \sum_{\alpha+\beta = j-1}  X_1^{\alpha} X_2^{\beta}
\]
where $ a_i $'s are coefficients of $ \f =  a_n x^n + a_{n-1} x^{n-1} + \cdots + a_0 $.
In this paper, we adopt the alternative definition as in \cite{Hu24}:
\[
 \oper_\p^{(r)}(X_1 , \cdots, X_{r}) \equiv \prod_{j=1}^{r-1} \oper_\p^{(2)}(X_j , X_{r}) \Mod \p(X_r) 
\]
using congruence relations. In Corollary \ref{eq:f=(t-zeta)n}, we show that the two definitions above are compatible.  
For any integer $ \lambda $, we define 
\[
    \HH_\lambda :  = \sum_{i=0}^n a_i \h_{i-1+\lambda}(X_1, \cdots, X_r),
\]
where $ \h_k(X_1,\dots,X_r)$ denotes the complete homogeneous symmetric polynomial of degree $k$.
Using the properties developed in this paper, one can derive that
\begin{equation*}
     \oper_\f^{(r)}(X_1,\cdots,X_r)  = \det \begin{pmatrix}
            \HH_0  & \HH_{1} & \cdots & \HH_{r-2} \\
            \HH_{-1} & \HH_0 & \cdots & \HH_{r-3} \\
            \vdots & {\vdots} & {\ddots}& {\vdots}\\
            \HH_{2-r} & \HH_{3-r} & \cdots & \HH_0 
        \end{pmatrix},
\end{equation*}
 which has not appeared in any reference.
\subsection{Main Results}
% Let $\f( t ) $ be a monic irreducible polynomial of degree $ n $.
The goal of this paper is to provide a  conceptual interpretation of the explicit formula \eqref{eq:weil0} for Weil pairings. Our approach is based on the theory of Anderson generating functions. For $ z $ contained in the lattice of $ \phi $, the Anderson generating function $ \omega_z $ of $ \phi $ is defined by 
\[
\omega_z(t) =\sum_{i=0}^\infty \exp_\phi(\frac{z}{\theta^{i+1}}) t^i.
\]
This function becomes a crucial tool connected with the special values of $L$-functions. We refer to \cite{MR1059938,MR2979866} for more details. Let $\T$ denote the Tate algebra. There exists a natural evaluation map 
\[ 
\ev_\f : \T \to \mathbb{C}_\infty[t] / \f (t)\mathbb{C}_\infty[t] \]
extending the quotient map $ \mathbb{C}_\infty[t] \to \mathbb{C}_\infty[t] / \f (t)\mathbb{C}_\infty[t]  $.  
The unique representative $ [\omega]_\f$ for $ \ev_\f (\omega)$ of minimal degree is called \textbf{the $\f$-remainder} of $\omega$. We show in Corollary \ref{cor:expansion} that the rank-two Weil operator relates the $\f$-remainder of $ \omega_z(t) $ to a simple expression:
\begin{equation}\label{omegaz}
    [\omega_z (t)]_\f= \oper_{\f}^{(2)}(\xx, t ) \diamond_{\phi}  \exp_{\phi}(\frac{z}{\f(\theta)  } )  ,
\end{equation}
where $\diamond_{\phi} $ denotes the Drinfeld action with respect to the variable  $ \xx $. 
In particular, the leading coefficient of $ [\omega_z (t)]_\f $ is identical to $ \exp_{\phi}(\frac{z}{\f (\theta)}) $ in the $\f $-torsion  of $\phi $. Now take $r$ generating functions $ \omega_{z_1} , \cdots, \omega_{z_r}$ corresponding to the elements $ \mu_1, \cdots, \mu_r $ of the $ \f $-torsion, i.e., 
\[ 
\exp_{\phi}(\frac{z_i}{\f (\theta)}) =  \mu_i.  \] 
It is shown in Corollary \ref{cor:basis} that all coefficients of $ [\omega_{z_i} (t)]_\f $ with $ 1\leqslant i \leqslant r$ form an $\mathbb{F}_q$-linear basis of the $\f $-torsion module $ \phi[\f ] $, analogous to the results in \cite{MR3338012, MaurischatPerkins2022, MR4395011 }.  Let $ \kappa(t) $ be the Moore determinant of $ \omega_{z_1} , \cdots, \omega_{z_r}$. As shown by Hamahata \cite{HY93}, $ \kappa(t) $  is exactly the Anderson generating function of $ \psi $. In the main theorem, Theorem \ref{thm:main}, we prove that 
the leading coefficient of $ [\kappa(t)]_\f $ is identical to the Weil pairing of $ \mu_1,\cdots, \mu_r $. 
 For the precise details we refer to the summary in Section \ref{sec:summary}.

The advantage of our approach is that we can easily arrive at an explicit formula for Drinfeld modules of arbitrary rank. This formula matches Katen's result and reveals further arithmetic properties.  Because of the nature of our approach, verifying properties (1)-(5) in Definition \ref{defn:weil} becomes completely elementary. 
Our technique of $ \f $-remainders is developed from \cite{MR3338012, MaurischatPerkins2022, MR4395011}, so one can easily recover some results therein. For instance, the Taylor coefficients of $ \omega_z $ are recovered in Corollary \ref{cor:MP}, and the alternative basis $ \{ C_{z_i,j} \} $ of torsion modules (analogous to the basis of Maurischat and Perkins) is found in Corollary \ref{cor:basis}.
The framework we have introduced suggests a possible generalization to Drinfeld modules over arbitrary Dedekind domains. The connection between Weil operators and Anderson generating functions in \eqref{omegaz} provides a useful blueprint for investigating such generalizations. This will be the scope of future research.

\subsection{Outline}
In Section \ref{sec:operator} we introduce the Weil operator and develop its basic properties. We give the definition of the rank-$r$ Weil operator and record functoriality and simple congruence relations that will be used later. Section \ref{sec:expansion}  introduces the $\f $-remainder of a function in $\T$, and its connection with Hasse-Schmidt derivatives.
% and exhibits the relation between its coefficients and the $\f $-torsion points.
Section \ref{sec:generating} recalls Anderson generating functions for a Drinfeld module and determine the $\f $-remainder of these
generating functions by applying the results in Section \ref{sec:expansion}.
In Section \ref{sec:weil}, we consider the Moore determinant of Anderson generating functions and relate them to Weil operators. The main result of the paper, Theorem \ref{thm:main}, is stated and proved there: the top coefficient of the Moore determinant of $r$ suitably chosen generating functions equals the Weil pairing of the corresponding $\f $-torsion points.  
\subsection{Notations}
\paragraph{\textbf{Fields and Rings}} 
\begin{itemize}
    \item $\mathbb{F}_q$ : Finite field with $q$ elements.
    \item $\xA = \mathbb{F}_q[\xx]$ : Polynomial ring in one variable over $\mathbb{F}_q$.
    \item $\A = \mathbb{F}_q[t]$ :  A copy of $ \xA $. 
    \item $K$ : A field containing $\mathbb{F}_q(\theta)$, with embedding $ \mathbb{F}_q[t] \to\mathbb{F}_q[\theta] \to K $.
    \item  $K^{\sep}$ : Separable closure of $K$.
    \item $K_\infty = \mathbb{F}_q((1/\theta))$ : Completion of $\mathbb{F}_q(\theta)$ at the infinite place.
    \item $\mathbb{C}_\infty$ : Completion of the algebraic closure of $K_\infty$.
\end{itemize}
  
\paragraph{\textbf{Drinfeld Modules}} 

\begin{itemize}
   \item $\phi$ : A Drinfeld module of rank $r$, often given by $\phi_\xx = \theta + g_1\tau + \cdots + g_r\tau^r$.

    \item $\psi$ : A rank-one Drinfeld module, often the exterior product of $\phi$ given by $\psi_\xx = \theta + (-1)^{r-1}g_r\tau$.

   \item  $\phi[\f]$ : The $\f$-torsion module of $\phi$, where $\f\in\xA$ is a polynomial.

   \item $\exp_\phi$ : The exponential map of $\phi$. \item $\Lambda_\phi$ : The kernel lattice of $\exp_\phi$. 
   \item $ D_i $ : The $i$-th coefficient in the expansion of $\exp_\phi$, i.e., $\exp_\phi(z) = \sum_{i=0}^\infty \frac{z^{q^i}}{D_i}$.

   \item $\diamond_{\phi}$ : The Drinfeld action: for $a(\xx) \in \mathbb{F}_q[\xx]$, $a(\xx) \diamond_{\phi} \mu = \phi_a(\mu)$.
\end{itemize}
\paragraph{\textbf{Polynomials and Ideals}}
\begin{itemize}
    \item $\p$ : A monic irreducible polynomial in $ \A $ (resp. $\xA$) of degree $ d$.

    \item $\n$, $\m$ : General monic polynomials in $ \A $ (resp. $\xA$).

    \item $\f  $ : A monic polynomial of degree $ n $, often with a factorization $\f = \m\n$.

    \item $\Mod \f(*)$ : Congruence modulo $ \f(X_1),\cdots, \f(X_n)$.

     \item $\h_k$ : complete homogeneous symmetric polynomial of degree $k$.
    
  \item $\HH_\lambda $ : A polynomial defined as $
    \HH_\lambda :  = \sum_{i=0}^n a_i \h_{i-1+\lambda}(X_1, \cdots, X_r)$.
    
\end{itemize}

\paragraph{\textbf{Weil Operators}}
\begin{itemize}
\item $\Omega$ : The K\"{a}hler differential module $\A dt$.

   \item  $\Omega_{\f}$ : The quotient module $\f^{-1}\Omega/\Omega$, isomorphic to $\A/\f \A $.

    \item $\langle\cdot,\cdot\rangle$ : The perfect pairing between $\A/\f\A $ and $\Omega_{\f}$.

   \item  $\Res_{\infty}$ : The residue at infinity.
    \item $\oper_\f^{(r)}(X_1,\dots,X_r)$ : The rank-$r$ Weil operator associated with the polynomial $\f$.

  \item $ [g(t)*\oper_\f^{(r)}]  $ : The unique polynomial in $ \mathbb{F}_q[X_1, \cdots, X_r ]_{<n} $  congruent to $ g(X_i)  \oper_\f^{(r)} $.

   \item $\mathbb{V}(\f_1, \cdots, \f_n) $ : Determinant of the $r\times r$ matrix whose $(i,j)$-entry is $\f_i(X_j)$.

  \item  $\mathbb{V}$ : Vandermonde determinant.

    \item $\D_\f$ : The dual map on $\A/\f\A$, defined by $\D_\f(t^i) = \sum_{j=0}^{\deg\f-i-1} a_{i+j+1} t^j$.

    \item $\eta_\f^*$ : The differential $-\frac{1}{\f} dt$.

    \item $ \delta_{i,j}$  :  Kronecker symbol.
\end{itemize}
\paragraph{\textbf{Tate algebra}}
\begin{itemize}

 \item $\T$ : The Tate algebra over $\mathbb{C}_\infty$, the ring of formal power series with coefficients tending to zero.
 \item $\langle\cdot,\cdot\rangle_\T$ : The pairing between $\T$ and $\Omega_{\f}$.
\item $ \ev_\f $ : The evaluation map from $ \T $ to $ \mathbb{C}_{\infty} [t] / \f  \mathbb{C}_{\infty} [t] $.
\item $ [\omega]_\f $ : The $\f$-remainder of $ \omega \in \T $. 
   \item $\delta_l^t$ : The $l$-th Hasse-Schmidt derivative with respect to $t$.
      
\end{itemize}
\paragraph{\textbf{Anderson Generating Functions}}
\begin{itemize}
    \item $\omega_z(t)$ : The Anderson generating function of $\phi$ for an element $z$ in the lattice.
\item $\omega^{(k)}$ : The $k$-th Frobenius twist of $\omega$.
    \item $ \kappa(t) $ : The Moore determinant of Anderson generating functions, which is the Anderson generating function of $\psi$.
    \item $C_{z,i}$ : The $i$-th coefficient in the $\f $-remainder of $\omega_z(t)$.

\end{itemize}

\paragraph{\textbf{Weil Pairing}}
\begin{itemize}

    \item $\Weil_\f$ : The Weil pairing, a multilinear map $\Weil_\f: \prod_{i=1}^r \phi[\f] \to \psi[\f]$.

    \item $\M$ : The Moore determinant, used in the definition of the Weil pairing and in Hamahata's construction.
\end{itemize}

\section{Weil Operator}\label{sec:operator}
We introduce the notions of Weil operators in this section.
\subsection{Dual Basis}

Let $\mathbb{F}_q$ be a finite field.  Let $\A = \mathbb{F}_q [ t ] $ be a polynomial ring over $ \mathbb{F}_q$. Let $\f = \f( t )$ be a fixed monic polynomial of the form 
\[
\f = a_n  t  ^n + a_{n-1}  t ^{n-1} + \cdots + a_0, 
\]
where $ a_n =1 $. Let $ \Omega $ be the K\"{a}hler differential module of $\A $ over $\mathbb{F}_q$, that is  $\Omega = \A \cdot d  t   $. Let $ \f^{-1} \Omega \subseteq \Omega \otimes_{\A} \mathbb{F}_q( t )  $ be the set of meromorphic differential $ \eta^* $ such that  $\f \cdot \eta^* \in \Omega$. It is evident that the quotient module 
$ \Omega_{\f} := \f^{-1} \Omega / \Omega $ is isomorphic to $  \A / \f \A $ as an $\A$-module. 
The pairing 
\[
    \A \otimes_{\mathbb{F}_q} \f^{-1}\Omega \to \mathbb{F}_q, \quad a \otimes \eta^*  \mapsto  \Res_{\infty} (a \eta^*) . 
\]
induces a perfect pairing 
\begin{equation} \label{eq:pairing}
      \langle -, -\rangle : \A /\f \A \otimes_{\A} \Omega_{\f} \to \mathbb{F}_q.
\end{equation} 
As a vector space, it is clear that 
\[
\A /\f \A = \langle 1,  t ,  t ^2, \cdots,  t ^{ n -1 } \rangle.
\]
We adopt the notation below to represent the dual basis of $ \A / \f \A $.
\begin{notation}
Denote by $\D_{\f} : \A / \f \A \to \A/ \f \A$ the $\mathbb{F}_q$-linear homomorphism:
\begin{equation}\label{eq:Df}
\D_{\f} ( t ^i ) = \sum_{j=0}^{n-i-1} a_{i+j+1}  t ^j.
\end{equation}
We call $\D_{\f} $ the dual map of $ \A/ \f \A $. 
\end{notation}
The following lemma is well-known, we give a proof to keep this paper self-contained
\begin{lem}{\label{lem:dual}}
With the notations above, we have
\begin{equation}\label{eq:dual}
  \langle t^i ,   \D_\f ( t ^j) \eta_\f^* \rangle = \delta_{i,j} ,  
  \end{equation}
  where $ \eta_\f^* = -\frac{1}{\f} dt $.
\end{lem}
\begin{proof}
 By the definition of pairing, we have  
 \[ 
        \langle  t ^{i}, \D_\f( t ^{j})\rangle =    \Res_{\infty}   \frac{- t ^{i} \D_\f( t ^{j})}{\f} d t  = \Res_{\infty}   \omega_{i,j}^*,
    \]
where $\omega_{i,j}^*=  \frac{- t ^{i} \D_\f( t ^{j})}{\f}dt $.  Taking $u=\frac{1}{ t }$, we find 
\[
\omega_{i,j}^*=   \dfrac{u^{j-i-1}g_j(u)}{g(u)}du, 
\]
where
\[
 g(u) =a_0 u^n+a_1u^{n-1}+\cdots+1 
\]
and
\[
g_j(u) =a_{j+1}u^{n-j-1}+a_{j+2}u^{n-j-2}+\cdots+1.
\]
We compute the residue of $ \omega_{i,j}^* $ at $u =0 $. It is evident that the valuation of $ \omega_{i,j}^* $ at $u =0 $ is given by $ j -i -1$.

    % Thus we have
    % \[\omega_{ij}^*=\dfrac{a_{j+1}u^{n-i-2}+a_{j+2}u^{n-i-3}+\cdots+u^{j-i-1}}{a_0u^n+a_1u^{n-1}\cdots+1}du\]

\begin{enumerate}
    \item Case $i<j$. Then $\omega_{i,j}^*$ is regular at $u=0$. So $\Res_{u=0}   \omega_{ij}^*=0$. 

\item Case $i=j$.  Note that $\dfrac{g_j(u)}{g(u)} \mid_{u=0} = 1 $.
So \[
\Res_{u=0} \omega_{i,j}^*= \Res_{u=0} \dfrac{g_j(u)}{g(u)}  \frac{du}{u} = \dfrac{g_j(u)}{g(u)} \mid_{u=0}  \Res_{u=0}\dfrac{du}{u} = 1. 
\]   
 
\item Case $i>j$.  Now we have
\begin{equation*}
    \omega_{ij}^*=\dfrac{u^{j-i-1}g_j(u)}{g(u)}du
    %=\dfrac{u^{j-i-1}(g(u)-h_j(u))}{g(u)}du
    =\dfrac{du}{u^{i+1-j}}- 
    \eta_{i,j}^* ,
\end{equation*}
 where
 \[ 
  \eta_{i,j}^*= \dfrac{u^{j-i-1} (a_0 u^n+a_1u^{n-1}+\cdots+a_ju^{n-j}) du}{g(u)}. \]
  Since the valuation of $ \eta_{i,j}^*$ at $ u = 0 $ is  $n-i-1 \geqslant 0 $, $ \eta_{i,j}^*$ is regular at $ u = 0 $. It follows $ \Res_{u=0} \eta_{i,j}^* = 0 $. Hence, we obtain 
 \[
 \Res_{u=0} \omega_{i,j}^* =- \Res_{u=0} \frac{du}{u^{i+1-j}} - \Res_{u=0} \eta_{i,j}^*  = 0 - 0 = 0. 
 \]
\end{enumerate}
Therefore, we obtain Equation \eqref{eq:dual}.
\end{proof}
In other words, the dual basis of $ \{ 1,  t ,  t ^2, \cdots,  t ^{n-1} \} $ is given by 
\[ \{ \D_\f(1) \eta^*_\f , \D_\f( t ) \eta^*_\f , \cdots, \D_\f( t ^{n-1})\eta^*_\f  \}  .
\]
\begin{lem}
For a polynomial $\p$ of degree $d $, the quotient algebra $ \A/ \p^k \A $ has an alternative basis 
 \[ \left\{ \e_{i,j}= \p( t )^i  t ^j \right\}_{ 0\leqslant i<k, 0\leqslant j<d} 
 \] whose dual basis is given by  
 \[ \left\{ \eta_{l,m}^* = \D_\p( t ^m)\p^{k-1-l}( t ) \eta_{\p^k }^* \right\}_{0\leqslant l<k,  0\leqslant m<d }.\]
\end{lem}
\begin{proof}
    It suffices to check the equation
 \begin{equation} \label{eq:eij} 
 \langle \e_{i,j} ,\eta_{l,m}^* \rangle  = \delta_{i,l} \delta_{j,m}. 
 \end{equation}
    By the definition of  the pairing \eqref{eq:pairing}, we have 
\begin{align*}
     \langle \e_{i,j}, \eta_{l,m}^* \rangle  
    & = \langle  \p^i  t ^j , \D_\p( t ^m)\p^{k-1-l}\eta_{\p^k }^*  \rangle \\
     & =\Res_{\infty}\left(\D_\p( t ^m)\p^{k-1-l}\p^i  t ^j\left(-\dfrac{1}{\p^k}\right) d t  \right) \\
    & =-\Res_{\infty}\left(\D_\p( t ^m) \p^{i-l-1}  t ^j d t  \right).
\end{align*}
We set $ \omega^*:= -\D_\p(t^m) \p^{i-l-1} t^j dt $.
It remains to compute the residue of $\omega^* $ at $  t  = \infty $.
 \begin{enumerate}
    \item Case $i> l$:  It is obvious that $ \omega^* $ is regular at all finite points. Applying the residue theorem for function fields, the residue of $ \omega^* $ at infinity vanishes. 

    \item Case $i <  l$:   The valuation of $ \omega^* $ is given by 
    \begin{align*}
        v_{\infty} (\omega^*) &= v_{\infty} \D_\p( t ^m) + (i-l-1) v_{\infty}\p + j v_{\infty}  t  + v_{\infty}(d t )  \\
        & =  - (d-1-m) - (i-l-1) d -j -2 \\
        & = - (i-l)d + (m-j-1) \geqslant 0. 
    \end{align*}
    It means that $ \omega^* $ is regular at $ \infty $. Therefore, $ \Res_{\infty} \omega^* = 0$. 
    \item Case $ i = l $. By Lemma \ref{lem:dual}, we have 
     \[ \Res_{\infty} \omega^* = -\Res_{\infty}\left(\D_\p( t ^m)\dfrac{ t ^j}{\p}\right) d  t =\delta_{ j,m} .
     \]
\end{enumerate}
In conclusion, we obtain Equation \eqref{eq:eij}, so the lemma follows.
\end{proof}

\subsection{Definition of Weil Operators}
Take $ n = \deg \f $. Denote by $ \mathbb{F}_{q}[X_1,\cdots, X_r]_{<n}  $ the collection of polynomials such that each $X_i $-degree is less than $ n$.   
Notice that $ \mathbb{F}_{q}[X_1,\cdots, X_r]_{<n}  $ is isomorphic to     
\[
\mathbb{F}_{q}[X_1,\cdots, X_r] / (\f(X_1), \cdots, \f(X_r) ) 
\]
as an $\mathbb{F}_q$-vector space.

\begin{defn}{\label{def:Ofr}}
We define rank-$r$ Weil operator associated with $ \f$ to be the polynomial
\[
\oper_\f^{(r)} (X_1, \cdots , X_{r}) \in  \mathbb{F}_{q}[X_1,\cdots, X_r]_{<n}
\]
as follows. 
Set $ \oper_\f^{(1)}(X_1) = 1 $, 
and \begin{equation}\label{eq:o2}
\oper_\f^{(2)} (X_1, X_2) = \sum_{k=0}^{n-1} \D_\f(X_1^k)X_2^{k} .
\end{equation}
For $r > 2 $, we define $\oper_\f^{(r)} (X_1, \cdots , X_{r}) $ to be the unique polynomial in $ \mathbb{F}_{q}[X_1,\cdots, X_r]_{<n}$ congruent to 
\begin{equation}\label{eq:prodo2}
  \prod_{j=1}^{r-1} \oper_\f^{(2)}(X_j , X_{r}) \Mod \f(X_r) .
\end{equation}
\end{defn}
\begin{prop}\label{prop:Weil-equiv}
    The rank-two Weil operator is given by 
\begin{equation}\label{eq:f2X1X2}
 \oper_\f^{(2)}(X_1, X_2) = \sum_{j=1}^{n} a_j \sum_{\alpha+\beta = j-1}  X_1^{\alpha} X_2^{\beta} =  \sum_{k=0}^{n-1} X_1^{k} \D_\f(X_2^k) . 
\end{equation}
In general, we have
 \begin{align*}
\oper_\f^{(r )} (X_1, \cdots , X_{r }) & \equiv  \sum_{0\leqslant j_1,\dots,j_r<n }   X_1^{j_1} \cdots X_{r-1}^{j_{r-1}}   \D_\f(X_{r}^{j_1}) \cdots \D_\f(X_{r}^{j_{r-1}})\Mod \f(X_r) \\
& \equiv  \sum_{0\leqslant j_1,\dots,j_r<n }   \D_\f( X_1^{j_1} ) \cdots \D_\f( X_{r-1}^{j_{r-1}} )  X_{r}^{j_1+ \cdots + j_{r-1}} \Mod \f(X_r). 
\end{align*}
\end{prop}

\begin{proof}
    Substituting $\D_\f$ (see Equation \eqref{eq:Df}) into Equation \eqref{eq:o2}, we have
\begin{align*}
    \oper_\f^{(2)}(X_1,X_2)&=\sum_{k=0}^{n-1} \D_\f(X_1^k)X_2^k\\
    &=\sum_{k=0}^{n-1}\sum_{j=0}^{n-k-1} a_{j+k+1} X_1^j X_2^k\\
    &=\sum_{j=1}^n a_j\sum_{\alpha+\beta=j-1} X_1^\alpha X_2^\beta.
\end{align*}
This confirms the first equality in \eqref{eq:f2X1X2}. The second one in \eqref{eq:f2X1X2} follows similarly. Finally, the congruence relation for  $ \oper_\f^{(r)} $ is obtained by applying \eqref{eq:f2X1X2}.

\end{proof}

\subsection{Properties of Weil Operators}
\begin{lem}\label{lem:X1X2equal}
    The rank-two Weil operator satisfies
    \[
    X_1 \oper^{(2)}_\f(X_1, X_2) - \f(X_1 )=    X_2 \oper^{(2)}_\f(X_1, X_2)- \f(X_2 ).
 \]
\end{lem}
\begin{proof}
  It is straightforward that
    \begin{align}
X_1\oper_\f^{(2)}&=X_1\sum_{j=1}^n a_j\sum_{\alpha+\beta=j-1} X_1^\alpha X_2^\beta \nonumber \\
        &=\sum_{j=1}^n a_j\sum_{\alpha+\beta=j-1} X_1^{\alpha+1} X_2^\beta \nonumber  \\
        &=\sum_{j=1}^n a_j X_1^j+\sum_{j=2}^n a_j \sum_{\alpha+\beta=j-2} X_1^{\alpha+1} X_2^{\beta+1} \label{eq:x1of}\\
        &=\f(X_1)-a_0+\sum_{j=2}^n a_j \sum_{\alpha+\beta=j-2} X_1^{\alpha+1} X_2^{\beta+1}.\nonumber 
    \end{align}
 Therefore, we have
  \begin{equation}\label{eq:x1}
  X_1\oper_\f^{(2)}-\f(X_1)=-a_0+\sum_{j=2}^n a_j \sum_{\alpha+\beta=j-2} X_1^{\alpha+1} X_2^{\beta+1}.
  \end{equation}
 In the same manner, we obtain 
 \begin{equation}\label{eq:x2}
 X_2\oper_\f^{(2)}-\f(X_2)=-a_0+\sum_{j=2}^n a_j \sum_{\alpha+\beta=j-2} X_1^{\alpha+1} X_2^{\beta+1}.\end{equation}
Combining Equations \eqref{eq:x1} and \eqref{eq:x2},  we finish the proof.
\end{proof}
In other words, we derived a new expression 
\begin{equation}\label{eq:operf}
    \oper_\f^{(2)} = \frac{\f(X_2) - \f(X_1) }{X_2 - X_1}. 
\end{equation}
From this expression, we know the rank-$r$ operators 
$\oper^{(r)}_\f$ are independent of the choice of the basis of $ \A / \f \A $.

\begin{notation}
If polynomials $\m$ and $\n$ in $ \mathbb{F}_q[X_1, \cdots, X_r]$ are congruent modulo $ (\f(X_1),\cdots , \f(X_r)) $ we write 
\[
    \m \equiv \n   \Mod \f(*) .
\]
\end{notation}

\begin{prop}{\label{prop:changeindex}}
Let $g(X_1,\cdots , X_r ) \in \mathbb{F}_q[X_1, \cdots, X_r ] $. Let  $ \beta_1,\cdots , \beta_r  \in \{ 1,\cdots, r \}$. 
Then 
\[ g(X_1,\cdots, X_r )  \oper_\f^{(r)} \equiv g(X_{\beta_1},\cdots, X_{\beta_r} ) \oper_\f^{(r)} \Mod \f(*). 
 \]
\end{prop}

\begin{proof}

We have shown that
\[
X_1\oper^{(2)}_\f\equiv X_2\oper^{(2)}_\f\Mod\f(*)
\] 
in Lemma \ref{lem:X1X2equal}.  For general $r \geqslant 2$ and $ 1\leqslant i\leqslant r$, we have
\begin{align*} 
    X_i \oper_\f^{(r)}&\equiv X_i\prod_{j=1}^{r-1}\oper^{(2)}_\f(X_j,X_r)\Mod\f(*)\\
    & \equiv  X_i\oper^{(2)}_\f(X_i,X_r)\prod_{j=1,j\ne i}^r \oper_\f^{(2)}(X_j,X_r)\Mod\f(*)\\
    &\equiv X_r\oper^{(2)}_\f(X_i,X_r)\prod_{j=1,j\ne i}^r \oper_\f^{(2)}(X_j,X_r)\Mod\f(*)\\
    &\equiv X_r\prod_{j=1}^{r-1}\oper^{(2)}_\f(X_j,X_r)\Mod\f(*)\\
    &\equiv X_r\oper^{(r)}_\f \Mod\f(*).
\end{align*}
It follows that for any $1\leqslant i<j\leqslant r$,
\begin{equation}\label{eq:Xij}
X_i\oper^{(r)}_\f \equiv X_j\oper^{(r)}_\f \Mod\f(*).
\end{equation}
Applying the relation \eqref{eq:Xij} recursively, we have
\begin{equation}\label{eq:xko}
    X^k_i\oper^{(r)}_\f  \equiv X_i^{k-1}X_j\oper^{(r)}_\f  \equiv \cdots \equiv  X_i X_j^{k-1} \oper^{(r)}_\f \equiv X_j^{k}\oper^{(r)}_\f \Mod\f(*) .
\end{equation}
Again, we use the relation \eqref{eq:xko} recursively to obtain  
\begin{equation*}
  X_1^{k_1}X_2^{k_2}\cdots X^{k_r}_r\oper^{(r)}_\f(X_1,\cdots,X_r)
  \equiv X^{k_1}_{\beta_1}X^{k_2}_{\beta_2}\cdots X^{k_r}_{\beta_r}\oper^{(r)}_\f(X_1,\cdots,X_r)\Mod\f(*).
\end{equation*}
We have proved the case when $ g(X_1,\cdots , X_r) $ is a monomial. The general case is concluded by the linearity of the congruence relation.
\end{proof}
\begin{defn}{\label{def:star}}
From the previous proposition, one may define $ [g(t)*\oper_\f^{(r)}]  $ to be the unique polynomial in $ \mathbb{F}_q[X_1, \cdots, X_r ]_{<n} $  equivalent to $ g(X_i)  \oper_\f^{(r)} $ for any $ i \in \{ 1, \cdots , r \} $. 
\end{defn}

\begin{example} 
    In this example, we provide the formula 
    \begin{equation}\label{eq:tkO}
    [ t^k*\oper_{\f}^{(2)} (X_1,X_2)] = \frac{X_2^k\f(X_1)-\f(X_2)X_1^k}{X_1-X_2} 
    \end{equation}
    for $ 0 \leqslant k \leqslant n-1 $. The right‑hand side clearly lies in $\mathbb{F}_q[X_1,X_2]_{<n}$. For $k=0$, the identity \eqref{eq:tkO} reduces to \eqref{eq:operf}. 
    For $k \geqslant 1$, substituting \eqref{eq:operf} yields 
    \begin{align*}
     \frac{X_2^k\f(X_1)-\f(X_2)X_1^k}{X_1-X_2}  &=  \frac{X_2^k \left( \f(X_1) - \f(X_2) \right) -\f(X_2) ( X_1^k-X_2^k )}{X_1-X_2} \\
     & = X_2^k \oper_{\f}^{(2)} (X_1,X_2) - \f(X_2) \sum_{\alpha + \beta = k-1 } X_1^{\alpha} X_2^{\beta} \\
     & \equiv X_2^k \oper_{\f}^{(2)} (X_1,X_2)  \Mod \f(X_2) .
    \end{align*}
   This completes the verification for $1\leqslant  k\leqslant n-1$.
\end{example}
We extend the expression \eqref{eq:prodo2} as follows.
\begin{prop}\label{prop:graph}
    Let $G$ be a connected undirected graph with $ r $ vertices labeled by $ \{ 1,\cdots, r\} $ and $ (r -1) $ edges labeled by    $ \{ (\alpha_i,\beta_i)\}_{ i = 1,\cdots , r-1 }  $. Then  
    \[
        \oper_\f^{(r) }(X_1,\cdots, X_r) \equiv \prod_{i=1}^{r-1}\oper_{\f}^{(2)} (X_{\alpha_i}, X_{\beta_i} ) \Mod \f(*).
    \]
\end{prop}
\begin{proof}
Let $ \mathcal{G} $ be the set of all connected undirected graphs with $ r $ vertices labeled by $ \{ 1,\cdots, r\} $ and $ (r -1) $ edges.
We define the map 
\[
    \Theta: \mathcal{G} \to \mathbb{F}_q[X_1, \cdots, X_r ]_{<n}
\]
sending $ G \in \mathcal{G} $ to the unique polynomial congruent to \[ \prod_{i=1}^{r-1}\oper_{\f}^{(2)} (X_{\alpha_i}, X_{\beta_i} ) \Mod \f(*), \]
where $ \{ (\alpha_i, \beta_i) \}_{i=1,\cdots, r-1} $ denotes the edges of $G $. Then the proposition is equivalent to saying that $ \Theta $ is a constant map. 

It is clear that for any $ G' , G'' \in \mathcal{G}$, there is a sequence $ G_1, \cdots, G_m$ such that 
\begin{enumerate}
    \item $ G_1 = G', G_m = G''$;
    \item For each $ k =1, \cdots, m-1$, the graphs $G_k$ and $ G_{k+1}$ differ by exactly one edge.
\end{enumerate}
To confirm the equality $ \Theta(G') = \Theta(G'') $, it suffices to prove $ \Theta(G_k) = \Theta(G_{k+1}) $ for each $k$. 
Let $ (\alpha_i', \beta_i')$ and $ (\alpha_i'', \beta_i'')$ be the edges of $G_k$ and $ G_{k+1}$ respective. 
By assumption, we can further assume that 
\begin{enumerate}
    \item For $i \geqslant 2$,
    $(\alpha_i',\beta_i')$ equals $(\alpha_i'',\beta_i'')$.
    \item The edges $ (\alpha_1', \beta_1') $ and $ (\alpha_1'', \beta_1'')  $ share the same vertex $ \alpha $. Without loss of generality, $\alpha_1' = \alpha_1'' =\alpha$. 
\end{enumerate}
Since $ G_k $ is connected, there exists a path in $ G_k $ connecting $ \beta_1' $ and $ \beta_1'' $. Write this path as  
\[
\beta_1' = \gamma_0, \gamma_1, \dots, \gamma_l, \gamma_{l+1} = \beta_1'',
\]  
where each consecutive pair $ (\gamma_{j-1}, \gamma_j) $ is an edge of $ G_k $. We refine the sequence $ \{G_i\} $ by inserting additional graphs $ G^{\gamma_j} $ for $ j = 0, 1, \dots, l+1 $, with $ G^{\gamma_0} = G_k $ and $ G^{\gamma_{l+1}} = G_{k+1} $. For each $ 1 \leqslant j \leqslant l $, the graph $ G^{\gamma_j} $ is obtained from $ G_k $ by replacing the edge $ (\alpha, \beta_1') $ with $ (\alpha, \gamma_j) $ (equivalently, from $ G_{k+1} $ by replacing $ (\alpha, \beta_1'') $ with $ (\alpha, \gamma_j) $). Thus, consecutive graphs in the refined sequence differ by exactly one edge; see Figure \ref{fig:mygraph}.
\begin{figure}[htbp]
    \centering
\begin{tikzpicture}[
    dot/.style={circle, fill=black, minimum size=5pt, inner sep=0pt},
    line/.style={gray, thick},
    font=\normalsize,
    scale=0.3,  
   xshift=10cm,
     % remember picture 
   % overlay  
    ]
\begin{scope}[xshift=-20cm]
\coordinate (A1) at (0, 0);
\coordinate (B1) at (5, 0.5);
\coordinate (D1) at (-0.5, 2);
\coordinate (E1) at (2.5, 6);
\coordinate (F1) at (5.5, 3);
\draw[line] (B1)--(F1);
\draw[line] (A1)--(D1)--(E1);
\draw[line, dashed] (A1)--(B1);
\node[dot, label={below:$\gamma_1$}] at (A1) {};
\node[dot, label={below:$\gamma_l$}] at (B1) {};
\node[dot, label={left:$\gamma_0$}] at (D1) {};
\node[dot, label={above:$\alpha$}] at (E1) {};
\node[dot, label={right:$\gamma_{l+1}$}] at (F1) {};
\node at (2.5,-1.5) {$G_k$};
\end{scope}

% 图2
\begin{scope}[xshift=-10cm]
\coordinate (A2) at (0, 0);
\coordinate (B2) at (5, 0.5);
\coordinate (D2) at (-0.5, 2);
\coordinate (E2) at (2.5, 6);
\coordinate (F2) at (5.5, 3);
\draw[line] (A2)--(E2);
\draw[line] (F2)--(B2);
\draw[line] (D2)--(A2);
\draw[line, dashed] (A2)--(B2);
\node[dot, label={below:$\gamma_1$}] at (A2) {};
\node[dot, label={below:$\gamma_l$}] at (B2) {};
\node[dot, label={left:$\gamma_0$}] at (D2) {};
\node[dot, label={above:$\alpha$}] at (E2) {};
\node[dot, label={right:$\gamma_{l+1}$}] at (F2) {};
\node at (2.5,-1.5) {$G^{\gamma_1}$};
\end{scope}

\begin{scope}[xshift=-8cm]
 \node at (9, 3) {$\cdots$};
\end{scope}
% 图3
\begin{scope}[xshift=6cm]
\coordinate (A3) at (0, 0);
\coordinate (B3) at (5, 0.5);
\coordinate (D3) at (-0.5, 2);
\coordinate (E3) at (2.5, 6);
\coordinate (F3) at (5.5, 3);
\draw[line] (B3)--(E3);
\draw[line] (B3)--(F3);
\draw[line] (D3)--(A3);
\draw[line, dashed] (A3)--(B3);
\node[dot, label={below:$\gamma_1$}] at (A3) {};
\node[dot, label={below:$\gamma_l$}] at (B3) {};
\node[dot, label={left:$\gamma_0$}] at (D3) {};
\node[dot, label={above:$\alpha$}] at (E3) {};
\node[dot, label={right:$\gamma_{l+1}$}] at (F3) {};
\node at (2.5,-1.5) {$G^{\gamma_l}$};
\end{scope}
\begin{scope}[xshift=16cm]
\coordinate (A4) at (0, 0);
\coordinate (B4) at (5, 0.5);
\coordinate (D4) at (-0.5, 2);
\coordinate (E4) at (2.5, 6);
\coordinate (F4) at (5.5, 3);
\draw[line] (E4)--(F4)--(B4);
\draw[line] (D4)--(A4);
\draw[line, dashed] (A4)--(B4);
\node[dot, label={below:$\gamma_1$}] at (A4) {};
\node[dot, label={below:$\gamma_l$}] at (B4) {};
\node[dot, label={left:$\gamma_0$}] at (D4) {};
\node[dot, label={above:$\alpha$}] at (E4) {};
\node[dot, label={right:$\gamma_{l+1}$}] at (F4) {};
\node at (2.5,-1.5) {$G_{k+1}$};
\end{scope}
\end{tikzpicture}
    \caption{Graphs $G^{\gamma_i}$}
    \label{fig:mygraph} 
\end{figure}

According to Proposition \ref{prop:changeindex}, we get 
\[
\oper_{\f}^{(2)} (\alpha, \gamma_j)\oper_{\f}^{(2)} (\gamma_j, \gamma_{j+1}) \equiv \oper_{\f}^{(2)} (\alpha, \gamma_{j+1})\oper_{\f}^{(2)} (\gamma_j, \gamma_{j+1}) \Mod \f(*)
\]
for $0 \leqslant j \leqslant l$. Except for the edges $(\alpha, \gamma_j) $ and $  (\alpha, \gamma_{j+1}) $ the graphs $ G^{\gamma_j} $ and $ G^{\gamma_{j+1}}$ are the same. 
Therefore, it follows from the definition of $ \Theta $ that   
\[ \Theta(G^{\gamma_j}) = \Theta(G^{\gamma_{j+1}}) \]
and then $\Theta(G_k)= \Theta(G^{\gamma_0}) = \Theta(G^{\gamma_{l+1}}) =\Theta(G_{k+1})$. This completes the proof.

\end{proof}

\begin{cor}{\label{cor:symmetric}}
  The Weil operators are symmetric.  
\end{cor}
\begin{proof} 
Let $ (\sigma_1,\sigma_2,\cdots,\sigma_r) $ be a permutation of $ (1,2,\cdots, r) $.
We adopt the notation as in the proof of Proposition \ref{prop:graph}. 
Consider the diagram $G_0 \in \mathcal{G}$ with the $ (r-1) $ edges 
\[
    (1,r), (2,r), \cdots, (r-1,r); 
\]
and 
$G_\sigma \in \mathcal{G}$ with the $ (r-1) $ edges
\[
    (\sigma_1,\sigma_r), (\sigma_2,\sigma_r), \cdots, (\sigma_{r-1},\sigma_r).
\]
Then by Proposition \ref{prop:graph},
\[
\oper_\f^{(r)}(X_1,\cdots, X_r) = \Theta(G_0) = \Theta(G_\sigma) = \oper_\f^{(r)}(X_{\sigma_1},\cdots, X_{\sigma_r}).
\]
This implies that $\oper_\f^{(r)}$ is a symmetric polynomial.
\end{proof}
 \begin{cor}\label{cor:recursive}
 \begin{enumerate}
     \item  The Weil operators $ \oper_\f^{(r)} $ satisfy the recursive formula:
    \begin{equation}\label{eq:rp1}
\oper_\f^{(r+1)} = \sum_{k=0}^{n-1} [\D_{\f}( t ^k)*\oper_\f^{(r)}] X_{r+1}^{k}.
\end{equation}
\item In particular, the coefficient of $ X_{r+1}^{n-1} $ in $ \oper_\f^{(r+1)} $ is $ \oper_\f^{(r)} $. 
\item Similarly, we have
\begin{equation}{\label{eq:frp1}}
    \oper_\f^{(r+1)}= \sum_{k=0}^{n-1} [  t ^k *\oper_\f^{(r)}] \D_\f(X_{r+1}^{k} )  .
\end{equation} 
 \end{enumerate} 
\end{cor}
\begin{proof}
   (1) Recall in Definition \ref{def:Ofr}, we have
    \[
    \oper_\f^{(r)}(X_1,\cdots, X_r)  \equiv \prod_{j=1}^{r-1} \oper_\f^{(2)} (X_j, X_{r})\Mod\f(*).
    \]
  It is straightforward to see that 
   \begin{align}
  \sum_{k=0}^{n-1} [\D_\f(t^k)*\oper_\f^{(r)}] X_{r+1}^{k} 
  % & \equiv  \sum_{k=0}^{n-1} \D_\f(X_1^k ) \prod_{j=1}^{r-1} \oper_\f^{(2)} (X_j, X_{r}) X_{r+1}^{k}\Mod \f(*)  \\
  &  \equiv \sum_{k=0}^{n-1} \D_\f(X_1^k ) X_{r+1}^{k} \prod_{j=1}^{r-1} \oper_\f^{(2)} (X_j, X_{r})\Mod \f(*)  \label{eq:DfOf} \\
  & \equiv 
  \oper_\f^{(2)} (X_1, X_{r+1}) \prod_{j=1}^{r-1} \oper_\f^{(2)} (X_j, X_{r}) \Mod \f(*) \nonumber \\
  & \equiv  \oper_\f^{(r+1)} \Mod \f(*), \nonumber  
   \end{align} 
where the last equality we use Proposition \ref{prop:graph}. Notice that the $X_i$-degrees of the left-hand side of \eqref{eq:DfOf} are less than $n$. So the congruence yields an equality essentially.

(2) From \eqref{eq:rp1},  the coefficient of $ X_{r+1}^{n-1} $ in $ \oper_\f^{(r+1)} $ equals $  [\D_{\f} ( t ^{n-1})*\oper_\f^{(r)}] = \oper_\f^{(r)}$.

(3) The proof is analogous to the assertion (1).
\end{proof}

\begin{lem}\label{lem:mn2}
     Assume that $ \f= \m \n $. Let $ \oper_\f^{(2)}, \oper_\m^{(2)}, \oper_\n^{(2)} $ be the Weil operators with modulus $ \f, \m, \n $ respectively. 
\begin{enumerate}
         \item The Weil operator $ \oper_\f^{(2)} $ is given by 
\begin{equation}{\label{eq:f=mn}}
      \oper_{\f}^{(2)}  = \m(X_1) \oper_{\n}^{(2)} + \n(X_2) \oper_{\m}^{(2)}.
\end{equation} 
\item In particular, if $ \f= (t- \zeta) \n $, then 
\begin{equation}\label{eq:operp}
  \oper_{\f}^{(2)}  = (X_1-\zeta) \oper_{\n}^{(2)} + \n(X_2).
\end{equation} 
\end{enumerate}
\end{lem}
\begin{proof}
(1)  By Equation \eqref{eq:operf}, the Weil operator \(\oper_\f^{(2)}\) is given by:  
\[
\oper_\f^{(2)}(X_1,X_2) = \frac{\f(X_1) - \f(X_2)}{X_1 - X_2} = \frac{\m(X_1)\n(X_1) - \m(X_2)\n(X_2)}{X_1 - X_2}.
\]
Applying Equation \eqref{eq:operf} again, the right-hand side of  \eqref{eq:f=mn} becomes
\begin{align*}
& \m(X_1) \oper_\n^{(2)}(X_1,X_2) + \n(X_2) \oper_\m^{(2)}(X_1,X_2)  \\
 =& \m(X_1) \cdot \frac{\n(X_1) - \n(X_2)}{X_1 - X_2} + \n(X_2) \cdot \frac{\m(X_1) - \m(X_2)}{X_1 - X_2} \\
 = &\frac{ \m \n (X_1)- \m \n (X_2) }{X_1 - X_2 } = \oper_{\f}^{(2)} (X_1, X_2).
\end{align*}
 Then we get \eqref{eq:f=mn}.  

(2) For the case \(\f( t ) = ( t  - \zeta) \n( t )\), we set \(\m( t ) =  t  - \zeta\). Then \(\oper_\m^{(2)}\) is just the constant \(1\). It follows from \eqref{eq:f=mn} that 
\[
\oper_\f^{(2)} = (X_1 - \zeta) \oper_\n^{(2)} + \n(X_2) \cdot 1 = (X_1 - \zeta) \oper_\n^{(2)} + \n(X_2),
\]
thus confirming \eqref{eq:operp}.  
\end{proof}

 We generalize Lemma \ref{lem:mn2} to any rank $r \geqslant 2 $ as follows.
 \begin{prop}{\label{prop:f=mn}}
      Assume that $ \f= \m\n $. 
The Weil operators $ \oper_\f^{(r)} $ satisfy the recursive formula 
\begin{align}{\label{eq:f=mn,opr}}
    \oper_{\f}^{(r)}(X_1,\cdots, X_r) \equiv &  \m(X_l)\oper_\n^{(2)}(X_l,X_r)\oper_{\f}^{(r-1)}(X_1,\cdots, \hat{X_l}, \cdots, X_r) \notag\\
    & + \prod_{j \neq l; j=1 }^r \n(X_j) \oper_{\m}^{(r)} (X_1, \cdots, X_r) \Mod\f(X_r)
\end{align}
for any index $l \in \{ 1,\cdots, r\}$.
 \end{prop}
 
\begin{proof}
     By the symmetric property in Corollary \ref{cor:symmetric}, we assume $l=1$ without loss of generality. From the expression \eqref{eq:prodo2}, we obtain 
\begin{equation}\label{eq:or}
     \oper_{\f}^{(r)}(X_1,\cdots, X_r) 
    \equiv \oper_\f^{(2)}(X_1,X_r)\oper_\f^{(r-1)}(X_2,\cdots,X_r)\Mod \f(X_r)
\end{equation}
    According to  \eqref{eq:f=mn},  we have
\begin{equation}\label{eq:orf2}
    \oper_\f^{(2)}(X_1,X_r)  
      = \m(X_1)\oper_\n^{(2)}(X_1,X_r)+\n(X_r)\oper_\m^{(2)}(X_1,X_r).
\end{equation}
Substituting \eqref{eq:orf2} into \eqref{eq:or}, yields
\begin{align}
    &\oper_{\f}^{(r)}(X_1,\cdots, X_r)\nonumber\\
    % \equiv &\oper_\f^{(2)}(X_1,X_r)\oper_\f^{(r-1)}(X_2,\dots,X_r)\Mod \f(X_r)\\
    \equiv&(\m(X_1)\oper_\n^{(2)}(X_1,X_r)+\n(X_r)\oper_\m^{(2)}(X_1,X_r))\oper_\f^{(r-1)}(X_2,\cdots,X_r)\Mod \f(X_r)\nonumber\\
      \equiv& \Delta+\m(X_1)\oper_\n^{(2)}(X_1,X_r)\oper_\f^{(r-1)}(X_2,\cdots,X_r) \Mod \f(X_r), \label{eq:or3}
\end{align}
where
\[\Delta=\n(X_r)\oper_\m^{(2)}(X_1,X_r)\oper_\f^{(r-1)}(X_2,\cdots,X_r).
\]
Applying Equation \eqref{eq:f=mn} again yields 
    \begin{equation}{\label{eq:f=mn2}}   \oper_\f^{(2)}(X_j, X_r)=\n(X_j)\oper_\m^{(2)}(X_j, X_r)+\m(X_r)\oper_\n^{(2)}(X_j, X_r).
    \end{equation}
Substituting from \eqref{eq:f=mn2} for $j = 2 , \cdots, r-1$, the term $ \Delta $ becomes 
 \begin{align}
     \Delta \equiv &\n(X_r)\oper_\m^{(2)}(X_1,X_r)\prod_{j=2}^{r-1} \oper_\f^{(r-1)}(X_j,X_r)  \Mod \f(X_r)\notag \\
     \equiv &\n(X_r)\oper_\m^{(2)}(X_1,X_r)\prod_{j=2}^{r-1}\left(  \n(X_j)\oper_\m^{(2)}(X_j,X_r)+\m(X_r)\oper_\n^{(2)}(X_j,X_r) \right) \Mod \f(X_r)\notag\\
     \equiv &\n(X_r)\oper_\m^{(2)}(X_1,X_r)\prod_{j=2}^{r-1}\n(X_j)\oper_\m^{(2)}(X_j,X_r)\Mod\f(X_r)\notag\\
     \equiv &\prod_{j=2}^r\n(X_j)\oper_\m^{(r-1)}(X_1,\cdots,X_r)\Mod\f(X_r). \label{eq:or4}
 \end{align} 
Combining \eqref{eq:or3} with \eqref{eq:or4}, we complete the proof.
\end{proof}
The following corollary implies that our Weil operator is identical to the one in Katen's paper. 
\begin{cor}
      Assume that $ \f= (  t -\zeta) \cdot \m$. Then we have
\begin{align}{\label{eq:f=(t-zeta)n}}
    \oper_{\f}^{(r)}(X_1,\cdots, X_r) = &  \m(X_l)\oper_{\f}^{(r-1)}(X_1,\cdots, \hat{X_l}, \cdots, X_r)\notag \\
    & + \prod_{j \neq l; j=1 }^r (X_j-\zeta) \oper_{\m}^{(r)} (X_1, \cdots, X_r)
\end{align}
for any index $l \in \{ 1,\cdots, r\}$.
\end{cor}
\begin{proof}
Notice that $\oper_{ t -\zeta}^{(2)}(X_1,X_2)=1$. Replacing $\n$ by $ t -\zeta$ in Equation \eqref{eq:f=mn,opr} yields \eqref{eq:f=(t-zeta)n} up to congruence. Since the $X_r$-degree of \eqref{eq:f=(t-zeta)n} is less than $\deg(\f)$, the congruence relation can be strengthened to an equality.
\end{proof}
 The following corollary can be applied to verify (5) of Definition \ref{defn:weil}.
 \begin{cor}
    Assume that $ \f$ admits the decomposition $ \f=   \m \n  $. Then the Weil operators $ \oper_\f^{(r)} $ and $ \oper_\n^{(r)} $ satisfy the relation 
\begin{equation}\label{eq:nr}
    [\n(t)^{r-1} * \oper_{ \m }^{(r)}(X_1,\cdots, X_r) ] \equiv  \oper_{\f}^{(r)}(X_1,\cdots, X_r) \Mod \m(*) .
\end{equation} 
\end{cor}
\begin{proof}
From Proposition \ref{prop:f=mn}, we obtain 
\[
\prod_{i=1}^{r-1}\n(X_i)   \oper_{ \m }^{(r)}(X_1,\cdots, X_r)   \equiv  \oper_{\f}^{(r)}(X_1,\cdots, X_r) \Mod \m(*).
\]
It is equivalent to Equation \eqref{eq:nr}.
\end{proof}
As a consequence, for $ \f = \p^k $, we have 
\begin{equation}\label{eq:pk}
    [\p^{(r-1)(k-1)} * \oper_\p^{(r)} ]  \equiv  \oper_{\p^k }^{(r)} \Mod \p(*). 
\end{equation}
\subsection{Explicit Expression of Weil operators} 
Firstly, we provide two elementary examples for Weil operators.
\begin{example}
   In this example, we assume that $ \f( t )=  t ^n $. From Proposition \ref{prop:Weil-equiv}, we have
\[
\oper_{ t ^{n}}^{(2)} (X_1, X_2)= \sum_{\alpha+ \beta= n -1 } X_1^{\alpha} X_2^{\beta}.
\]
It follows from \eqref{eq:o2} that 
\begin{align}
    \oper_{ t ^{n}}^{(r)}(X_1,\cdots,X_r) & \equiv \oper_{ t ^{n}}^{(2)}(X_1, X_r)\oper_{ t ^{n}}^{(2)}(X_2, X_r)\cdots\oper_{ t ^{n}}^{(2)}(X_{r-1}, X_r)\Mod X_r^n \\
    & \equiv \sum_{k_1+ k_2 + \cdots + k_r = (n -1)(r-1)\atop0\leqslant k_i<n,1\leqslant i<r } X_1^{k_1} X_2^{k_2}\cdots X_r^{k_r}\Mod X_r^n.
\end{align}
Notice that the index $ k_r $ satisfies $ 1 \leqslant k_r \leqslant (n-1) (r-1)$ and for $ k_r \geqslant n $,
\[ X_1^{k_1} X_2^{k_2}\cdots X_r^{k_r}\equiv 0\Mod  X_r^n . \]
 So we have
\[\oper_{t^n}^{(r)}(X_1,\cdots,X_r)=\sum_{k_1+ k_2 + \cdots + k_r = (n -1)(r-1)\atop0\leqslant k_i<n,1\leqslant i\leqslant r } X_1^{k_1} X_2^{k_2}\cdots X_r^{k_r}.\]
This formula coincides with the one in \cite{P23}*{Exercise 3.7.3}.
\end{example}
\begin{example}
 We demonstrate the explicit formula for the rank-three Weil operator by applying Equation \eqref{eq:frp1}, i.e., 
 \[ \oper_\f^{(3)}(X_1,X_2,X_3)
    = \sum_{k=0}^{n-1}  [  t ^k*\oper_{\f}^{(2)}(X_1,X_2)] \D_\f(X_3^k).  
\] 
The following expression is derived by substituting from the equalities \eqref{eq:operf} and \eqref{eq:tkO}:  
\begin{align}
     &{}  \oper_\f^{(3)}(X_1,X_2,X_3) \nonumber \\ 
    =&\sum_{k=0}^{n-1} \frac{X_2^k\f(X_1)-\f(X_2)X_1^k}{X_1-X_2}  \left( \sum_{j=0}^{n-k-1} a_{k+j+1} X_3^j\right)\nonumber \\
    = &  - \frac{\f(X_2)}{X_1-X_2}\frac{\f(X_1)-\f(X_3)}{X_1-X_3}+\frac{\f(X_1)}{X_1-X_2}\frac{\f(X_2)-\f(X_3)}{X_2-X_3} \notag\\
     % = & \frac{1}{X_2-X_1}\left(\f(X_2)\frac{\f(X_1)-\f(X_3)}{X_1-X_3}-\f(X_1)\frac{\f(X_2)-\f(X_3)}{X_2-X_3}\right )\notag\\
     =&  \frac{\f(X_1)\f( X_2)  }{(X_1- X_3)(X_2- X_3)}+\frac{\f(X_2)\f( X_3)  }{(X_2- X_1)(X_3- X_1)}+\frac{\f(X_3)\f( X_1)  }{(X_3- X_2)(X_1- X_2)}\label{eq:O3}.
\end{align}
From Equation \eqref{eq:O3}, we see that  $\oper_\f^{(3)}(X_1,X_2,X_3)$ is symmetric.
 \end{example}
 
% \begin{example}
%  We demonstrate the explicit formula for the rank-three Weil operator by applying Equation \eqref{eq:frp1}, i.e., 
%  \[ \oper_\f^{(3)}(X_1,X_2,X_3)
%     = \sum_{k=0}^{n-1}  [  t ^k*\oper_{\f}^{(2)}(X_1,X_2)] \D_\f(X_3^k).  
% \] 
% The following expression is derived by substituting from the equalities \eqref{eq:Df} and \eqref{eq:tkO}: 

% \begin{align*}
%     &\oper_\f^{(3)}(X_1,X_2,X_3)\\ 
%     =&\sum_{k=0}^{n-1}\left(-\sum_{i=0}^{k-1} a_i\sum_{\alpha+\beta=k-1-i} X_1^{\alpha+i}X_2^{\beta+i}+\sum_{i=k+1}^n a_i\sum_{\alpha+\beta=i-k-1} X_1^{\alpha+k}X_2^{\beta+k}\right)\left( \sum_{j=0}^{n-k-1} a_{k+j+1} X_3^j\right)\\
%     = &{}  \sum_{k=0}^{n-1}\sum_{i=k+1}^{n}\sum_{j=0}^{n-k-1}\sum_{\alpha=0}^{i-k-1} a_i a_{k+j+1} X_1^{\alpha+k} X_2^{i-1-\alpha} X_3^j \\
%      &{} -\sum_{k=0}^{n-1}\sum_{i=0}^{k-1}\sum_{j=0}^{n-k-1}\sum_{\alpha=0}^{k-1-i} a_i a_{k+j+1} X_1^{\alpha+i} X_2^{k-1-\alpha} X_3^j .
% \end{align*}

% A direct computation shows that 
% \begin{align*}
%     \oper_\f^{(3)}(X_1,X_2,X_3) &=  \frac{\f(X_1)\f( X_2)  }{(X_1- X_3)(X_1- X_3)}+\frac{\f(X_1)\f( X_3)  }{(X_1- X_2)(X_3- X_2)}+\frac{\f(X_2)\f( X_3)  }{(X_2- X_3)(X_2- X_3)}.
% \end{align*}
%  \end{example}
In order to compute the rank-$r$ Weil operator, we introduce the following notations.
\begin{notation}
\begin{enumerate}
    \item Let $\mathbb{V}(\f_1,\cdots,\f_r)$ denote the determinant of the $r\times r$ matrix whose $(i,j)$-entry is $\f_i(X_j)$:
    \[
        \mathbb{V}(\f_1,\cdots,\f_r) := \det\begin{pmatrix}
            \f_1(X_1) & \f_1(X_2) & \cdots & \f_1(X_r) \\
            \f_2(X_1) & \f_2(X_2) & \cdots & \f_2(X_r) \\
            \vdots   & \vdots    & \ddots & \vdots   \\
            \f_r(X_1) & \f_r(X_2) & \cdots & \f_r(X_r)
        \end{pmatrix}.
    \]
    Notice that $\mathbb{V}:=\mathbb{V}(1,t,\cdots,t^{r-1})$ is the classical Vandermonde determinant:
\[
\mathbb{V}(1,t,\cdots,t^{r-1}) = \prod_{1 \leqslant i < j \leqslant r} (X_j - X_i).
\]

    \item Denote by $\h_k  = \h_k(X_1,\cdots,X_r)$ the complete homogeneous symmetric polynomial of degree $k$:
    \[
        \h_k = \sum_{\substack{i_1+\cdots+i_r = k \\ i_1,\dots,i_r \geqslant 0}} X_1^{i_1}\cdots X_r^{i_r}.
    \]
    For convenience, we set $\h_k = 0$ for $k < 0$.
\end{enumerate}
\end{notation}
Applying the Jacobi-Trudi identity \cite{MR3443860}, we have the following lemma.
\begin{lem}
    Write each polynomial as $ \f_k(t)=\sum_{i\geqslant 0} a_{k,i}t^i$.
Then
\[
\frac{\mathbb{V}\bigl(\f_1 , \f_2 , \cdots, \f_r \bigr)}{\mathbb{V} }
= \det_{1\leqslant k,j\leqslant  r}\Bigl( \sum_{i\geqslant 0} a_{k,i}\; \h_{i-k+j}(X_1,\cdots,X_r) \Bigr).
\]
\end{lem}
\begin{proof}
   For monomials $t^{n_1},\cdots,t^{n_r}$, the Jacobi-Trudi identity gives
\begin{equation}\label{eq:JT}
    \frac{\mathbb{V}(t^{n_1},\cdots,t^{n_r})}{\mathbb{V}}
= \det_{1\leqslant i,j\leqslant r}\bigl( \h_{n_i - i + j} \bigr).
\end{equation}

The left‑hand side of the lemma is multilinear in the rows
$( \f_k(X_1),\cdots,\f_k(X_r))$. Expanding each $ \f_k$ as a linear combination of monomials and using multilinearity, we obtain
\[
\frac{\mathbb{V}(\f_1,\cdots,\f_r)}{\mathbb{V}}
= \sum_{i_1,\cdots,i_r\geqslant 0} \biggl(\prod_{k=1}^r a_{k,i_k}\biggr)
\frac{\mathbb{V}(t^{i_1},\cdots,t^{i_r})}{\mathbb{V}}.
\]

Applying \eqref{eq:JT} to each term,
\begin{align*}
\frac{\mathbb{V}(\f_1,\cdots,\f_r)}{\mathbb{V}}
& = \sum_{i_1,\cdots,i_r\geqslant 0} \biggl(\prod_{k=1}^r a_{k,i_k}\biggr)
\det_{1\leqslant k,j\leqslant r}\bigl( h_{i_k - k + j} \bigr)\\ 
&= \det_{1\leqslant k,j\leqslant r}\Bigl( \sum_{i_k\geqslant 0} a_{k,i_k}\; h_{i_k - k + j} \Bigr).
\end{align*}
Renaming the summation index $i_k$ as $i$ completes the proof.
\end{proof}
Substituting $ \f_1(t) =1 $ and $ \f_i(t)= t^{i-2} \f(t)  $ for $i\geqslant 2$, we obtain the following result.
\begin{cor}\label{cor:Toeplitz}
Let $ \f = \sum_{i=0}^n a_i t^i $ be a polynomial and define for any integer $ \lambda $
\[
    \HH_\lambda :  = \sum_{i=0}^n a_i \h_{i-1+\lambda}(X_1, \cdots, X_r).
\]
Then we obtain
  \[
      \frac{\mathbb{V}(1, \f(t), t \f(t), \cdots,  t^{r-2}  \f(t))}{\mathbb{V}  } =  \det\begin{pmatrix}
            \HH_0  & \HH_{1} & \cdots & \HH_{r-2} \\
            \HH_{-1} & \HH_0 & \cdots & \HH_{r-3} \\
            \vdots & {\vdots} & {\ddots}& {\vdots}\\
            \HH_{2-r} & \HH_{3-r} & \cdots & \HH_0 
        \end{pmatrix},
    \]
where the right-hand side is an $(r-1) \times (r-1)$ Toeplitz determinant.
\end{cor} 
We expand $\mathbb{V}(1, \f(t), t \f(t), \cdots ,t^{r-2}  \f(t))$ along the first row and use the Vandermonde determinant:
\begin{align*}
 \mathbb{V}(1, \f(t), t \f(t), \cdots,  t^{r-2}  \f(t))&=\sum_{j=1}^r (-1)^{1+j}\left(\prod_{ i=1,i \neq j}^r \f(X_i )\right) \prod_{1\leqslant k<l\leqslant r\atop k,l \ne j} (X_l -X_k)\\
 &=\prod_{1\leqslant k<l\leqslant r} (X_l - X_k)\sum_{j=1}^r\prod_{i=1,i \neq j}^r\dfrac{\f(X_i)}{X_i-X_j}\\
 &= \mathbb{V}\cdot \sum_{j=1}^r\prod_{i=1,i\neq j}^r\dfrac{\f(X_i)}{X_i-X_j}.
\end{align*}
So we have
\begin{equation}\label{eq:special-Jacobi-Trudi}
    \dfrac{ \mathbb{V}(1, \f(t), t \f(t), \cdots,  t^{r-2}  \f(t))}{\mathbb{V}}=\sum_{j=1}^r\prod_{i=1,i\ne j}^r\dfrac{\f(X_i)}{X_i-X_j}.
\end{equation}

Motivated by the explicit formulas for $\oper_\f^{(2)}$ and $\oper_\f^{(3)}$, we obtain an explicit expression for the rank-$r$ operator.

\begin{thm}
The rank-$r$ operator $\oper_\f^{(r)}$ is given by 
\begin{equation}\label{eq:Orf_i}
\oper_\f^{(r)}(X_1, \cdots ,X_r) = \sum_{j=1}^r\prod_{i=1,i\neq j}^r\dfrac{\f(X_i)}{X_i-X_j},
\end{equation}
or equivalently by the determinant in Corollary \ref{cor:Toeplitz}.
\end{thm}

\begin{proof}
Let $\Theta_{\f}^{(r)}(X_1,\cdots, X_r)$ denote the right-hand side of \eqref{eq:Orf_i}. 
% It is straightforward to check that 
% \[
% \Theta^{(r)}(X_1, \cdots, X_r )= \frac{\mathbb{V}(1, \f(t), t \f(t), \cdots, t^{r-2} \f(t))}{\mathbb{V} }.
% \]
From Corollary \ref{cor:Toeplitz} and Equation \eqref{eq:special-Jacobi-Trudi}, we know $\Theta_{\f}^{(r)}$ is symmetric and belongs to $\mathbb{F}_q[X_1, \cdots, X_r]_{<n}$.

Note that $\oper_\f^{(r)} = \Theta_{\f}^{(r)}$ is already known for $r=1$ and $r=2$.
For $r \geqslant 3$,   we have the congruence relation in $\mathbb{F}_q(X_1, \cdots, X_{r-1})[X_r] $:
\begin{align*}
    \Theta_{\f}^{(r)}(X_1, \cdots, X_r ) &= \sum_{j=1}^r \prod_{i=1, i\neq j}^r\frac{ \f(X_i)}{ X_i - X_j}\\
    &\equiv \prod_{i=1}^{r-1}\dfrac{\f(X_i)-\f(X_r)}{X_i-X_r}\Mod \f(X_r)\\
    &\equiv \prod_{i=1}^{r-1} \oper_\f^{(2)}(X_i, X_r) \Mod \f(X_r).
\end{align*}
According to Gauss lemma, the congruence relation is valid in $ \mathbb{F}_q[X_1, \cdots,X_r]$. Hence the theorem follows by definition.
\end{proof}
\begin{example}
    In particular, the rank-three Weil operator is given by  
    \begin{align*}
    \oper_{\f}^{(3)} & = \det \begin{pmatrix}
        \sum_{i=0}^{n} a_{i} \h_{i-1} &\sum_{i=0}^{n} a_{i} \h_{i} \\
        \sum_{i=0}^{n} a_{i} \h_{i-2} &\sum_{i=0}^{n} a_{i} \h_{i-1}
    \end{pmatrix} \\
    & =  \sum_{i=0}^{n} \sum_{j=0}^{n}  a_{i}a_{j} \left(\h_{i-1} \h_{j-1} - \h_{i} \h_{j-2} \right) \\
    & =   \sum_{i=0}^n \sum_{j=0}^n a_i a_j (\h_{i-1} \h_{j-1} - \frac{1}{2} \h_{i} \h_{j-2}- \frac{1}{2} \h_{i-2} \h_{j} ) . 
    \end{align*}
\end{example} 
The special case $ n = 3 $ coincides with \cite{KJ21}*{Example 3.4}.

\section{The $\f$-remainder of Tate Algebra}\label{sec:expansion}
%From now on, we assume that $\p $ be an irreducible polynomial of degree $d $. We denote by $ \zeta_j := \zeta^{q^{j-1}} $ with $j = 1, \cdots, d$ the roots of $ \p $. 

%In this section, we expand the generating function $ \omega_z $ in a proper way. Most results in this section are modified from \cite{MaurischatPerkins2022}. The main difference is that we adopt the notion of $ \p^k $-remainder, rather than the Taylor coefficients of the Hasse-Schmidt derivatives of $ \omega_z $.  
%Roughly speaking, $ \p^k $-remainder is a notion generalizing equivalence modulo $\p^k $ to the case of Tate algebra.  
%For a  function $ \omega(t)  $, we define the $\p^{k}$-remainder of $ \omega(t)  $ as 
%\[
%    \omega(t) = \sum_{i=0}^{d k-1} a_i t^i + o(\p^{k})
%\]
%where $o(\p^{k})$ is divisible by $\p^{k}$. This definition is available when $\omega(t)$ is a rational function, but it is not clear for infinite series. Indeed, it is hard to understand an infinity series is divisible by the polynomial $\p^{k}$. 
%To handle this, we need to apply the Hasse-Schmidt derivatives to $\omega(t) $. 

\subsection{Tate Algebra and $\f$-Remainder}

Let \(\theta\) be an indeterminate over $\mathbb{F}_q$. Set \(K := \mathbb{F}_q(\theta)\) and \(K_\infty := \mathbb{F}_q((1/\theta))\). Denote by \(C_\infty\) the completion of an algebraic closure of \(K_\infty\), equipped with the canonical extension \(|\cdot|\) of the absolute value that makes \(K_\infty\) complete and is normalized by \(|\theta| = q\).

\begin{defn}The Tate algebra is defined as
\[ \T=\Big \{ \sum_{i\geqslant 0} c_i t^i\in\mathbb{C}_\infty [\![t]\!
] \mid  \abs{c_i} \rightarrow 0 \Big\}.\]
\end{defn}
%In particular, $ \n^{-1} \mathbb{C}_{\infty}[t] $ are contained in $ \T $ for any polynomial $ \n $ such that all roots $ \nu_i $ of $ \n  $ satisfying 
% $ \abs{\nu_i} > 1 . $
Let $ \f = \f(t)  $ be a monic polynomial of degree $ n $. The quotient algebra $\mathbb{C}_{\infty}[t] / \f \mathbb{C}_{\infty}[t]$ carries the structure of a Banach algebra when equipped with the norm
\[
\Big\| \sum_{i=0}^{n-1} g_i \bar{t}^i \Big\| = \max_{0 \leqslant i < n} |g_i|,
\]
where $\bar{t}$ denotes the class of $t$ in the quotient. In particular, $|\bar{t}| = 1$.
\begin{defn}
From the universal property of Tate algebras, the quotient morphism $ \mathbb{C}_{\infty}[t]  \to \mathbb{C}_{\infty}[t] / \f \mathbb{C}_{\infty}[t] $ can be naturally lifted to a morphism 
\[
 \ev_\f:  \mathbb{T} \to  \mathbb{C}_{\infty}[t] / \f \mathbb{C}_{\infty}[t] , \quad \omega(t) \mapsto \omega(\bar{t}). 
\]
The polynomial in $ \mathbb{C}_\infty[t]$ of degree $ < n $ representing the image of $ \omega(t) $ is called the $\mathfrak{f}$-remainder of $\omega(t) $, denoted by $ [\omega(t)]_\f $.

\end{defn}
It is trivial to see that the kernel of
$ \ev_\f:  \T  \to \mathbb{C}_{\infty}[t] / \f \mathbb{C}_{\infty} $ is $ \f \cdot \mathbb{T} $. 
% We have the congruences:
% \[
% [ \omega_1(t) \omega_2(t) ]_\f \equiv  [ \omega_1(t) ]_\f [\omega_2(t) ]_\f \Mod \f. \]
\begin{lem}\label{lem:Lagrange} 
    Suppose that  $ \p $ is an irreducible polynomial with roots $\zeta_1,\cdots, \zeta_d $. Then there exist coefficients $ a_i \in \mathbb{C}_{\infty} $ such that 
\[
  \omega (\zeta_j) = \sum_{i=0}^{d-1} a_i \zeta_j^i    . 
\]
Moreover, the polynomial $ \sum_{i=0}^{d-1} a_i t^i $  is the $\p$-remainder of $ \omega(t)$. 
\end{lem}

\begin{proof}
    The lemma follows from the Lagrange interpolation formula. 
\end{proof}

\subsection{Pairing between Tate Algebra and Differentials}

\begin{defn}\label{defn:pairingTate}
The pairing of \eqref{eq:pairing} extends naturally to 
\[
\langle -,- \rangle_\T : \T \otimes_{\mathbb{F}_q} \Omega_{\f} \to \mathbb{C}_\infty. 
\]
For $ \omega \in \T$, $ \eta^* \in \Omega_{\f }$, we define 
\[
\langle \omega, \eta^* \rangle_\T := \langle  [\omega(t) ]_\f, \eta^* \rangle_\T  := \Res_{\infty} [\omega(t) ]_\f\cdot \eta^* . 
\]
\end{defn}
%Assume that $ C_{j} $ are coefficients of $\f$-remainder of $ \omega(t)$, i.e., 
%\[
%[\omega(t) ]_\f = \sum_{j=0}^{n-1} C_j t^j  .
%\]
%For $ \eta^* \in \Omega_{\f} $, it is evident that
%\[
%\langle \omega , \eta^* \rangle_{\T}=\Res_{\infty} \sum_{j=0}^{n-1} C_j t^j \eta^*   = 
%\sum_{j=0}^{n-1} C_j \langle  t^j , \eta^*  \rangle =   \langle [\omega(t)]_\f, \eta^*  \rangle_\T .
%\]
The following lemma is obvious.
\begin{lem}\label{lem:pro}
The pairing $ \langle -, -\rangle_{\T} $ verifies the following properties:
 \begin{enumerate}
    \item For $a,b \in \mathbb{F}_q$, $ \omega_1, \omega_2 \in \T $,  and $ \eta^* \in \Omega_{\f} $, we have
    \[\langle a\omega_1+b\omega_2,\eta^*\rangle_{\T}=a\langle \omega_1,\eta^*\rangle_{\T}+b\langle \omega_2,\eta^*\rangle_{\T};
    \]
    \item For  $a,b \in \mathbb{F}_q$,  $ \eta_1^*, \eta_2^* \in \Omega_{\f} $,  and $ \omega \in \T $, we have
    \[\langle  \omega, a \eta_1^*+ b \eta_2^* \rangle_{\T}=a\langle \omega,\eta_1^*\rangle_{\T}+b\langle \omega,\eta_2^*\rangle_{\T};
    \]
    \item For $g(t) \in \mathbb{F}_q[t] $, 
    \[\langle g(t) \omega,\eta^* \rangle_{\T}= \langle   \omega,g(t) \eta^* \rangle_{\T}.\]
\end{enumerate}
\end{lem}
\begin{prop}\label{prop:coefficients}
Using the pairing in Definition \ref{defn:pairingTate}, the coefficients of $ \omega(t)$ can be directly written as 
\[
    C_i = \langle \omega, \D_{\f}(t^i)  \eta_{\f}^*\rangle_{\T}.
\]
In particular, as $ \D_{\f}(t^{n -1}) = 1$, we obtain 
\[
    C_{n-1} = \langle \omega,   \eta_{\f}^*\rangle_{\T}. 
\]
\end{prop}
\begin{proof}

Suppose the $\f$-remainder of $\omega$ is $\sum_{j=0}^{n-1}C_jt^j$. Using Lemma \ref{lem:dual}, we have
\[
 \langle \omega,   \D_{\f}(t^i) \eta_{\f}^* \rangle_{\T} = \langle \sum_{j =0 }^{n -1 } C_j t^j,   \D_{\f}(t^i) \eta_{\f}^* \rangle_{\T} =  \sum_{j =0 }^{n -1 } C_j\langle  t^j,   \D_{\f}(t^i) \eta_{\f}^* \rangle= C_i . \]
 \end{proof}

We can determine the pairing of rational functions by taking residue as in \eqref{eq:pairing}. 
\begin{lem}\label{lem:residue}
Let $ \n \in \mathbb{C}_\infty[t]$ be a polynomial prime to $ \f $. Assume that $ \nu_i $'s are the roots of $ \n $ with $|\nu_i|>1$ and $\zeta_j$'s are the roots of $\f$. For $\omega \in   \n^{-1} \mathbb{C}_{\infty}[t] \subseteq \T $ and $\eta^* \in \Omega_{\f}$, we have
\[
  \langle \omega, \eta^* \rangle_{\T} = \Res_{\infty} (\omega \eta^*) + \sum_i \Res_{\nu_i }(\omega \eta^*) = - \sum_{j}\Res_{\zeta_j}(\omega \eta^*). 
\]
\end{lem}
\begin{proof}
 Applying the residue theorem for function fields, we obtain 
\begin{align*}
    \Res_{\infty} (\omega \cdot\eta^* ) & = \Res_{\infty} ([\omega(t)]_\f \cdot \eta^* ) + \Res_{\infty} (\omega- [\omega(t)]_\f) \eta^* \\
    & = 
    \langle  \omega   , \eta^* \rangle_{\T} - \sum_i \Res_{\nu_i} (\omega- [\omega(t)]_\f) \eta^* - \sum_j\Res_{\zeta_j} (\omega- [\omega(t)]_\f) \eta^*.
\end{align*}
 The difference $ (\omega- [\omega(t)]_\f) $ is divisible by $\f $, so the differential $(\omega- [\omega(t)]_\f) \cdot \eta^*$ is regular at each $ \zeta_j $. Since $ \nu_i $ is not a pole of $\eta^*$, we see that $ [\omega(t)]_\f \cdot \eta^*$ is regular at $ \nu_i $.
 Thus, 
\[
\Res_{\nu_i}  [\omega(t)]_\f \cdot \eta^* = 
\Res_{\zeta_j}  (\omega- [\omega(t)]_\f) \cdot \eta^* = 0 .
\]
So the lemma follows.
\end{proof}

\begin{lem}\label{lem:fremainder}
     The $\f $-remainder of $ (\theta -t)^{-1} $ is 
    \[
    [\frac{1}{\theta -t }]_\f = 
    \sum_{i=0}^{ n-1} \frac{\D_{\f}(\theta^{i})}{\f(\theta)} t^{i} = 
    \frac{1}{\f(\theta)} \frac{ \f(\theta) - \f(t) }{\theta - t }  .
    \]
\end{lem}
\begin{proof}
Assume that the $\f$-remainder of $ (\theta -t)^{-1} $ is written as 
\begin{equation}\label{eq:Ci}  
[\frac{1}{\theta -t }]_\f = \sum_{i=0}^{n-1} C_i t^i .
\end{equation}
Since $ (\theta -t)^{-1}  \in (t - \theta)^{-1} \A$, we have by Proposition \ref{prop:coefficients} and Lemma \ref{lem:residue} that 
\begin{align*}
    C_i & = \langle \frac{1}{\theta -t } , \D_{\f}(t^i)\eta_\f^* \rangle_{\T} \\
    & = \Res_\infty \frac{1}{\theta -t }  \D_{\f}(t^i)\eta_\f^* + \Res_{\theta}\frac{1}{\theta -t }  \D_{\f}(t^i)\eta_\f^*. 
\end{align*}
The valuation of $\frac{1}{\theta -t }  \D_{\f}(t^i)\eta_\f^*$ at $ t = \infty  $ is
\[
1 - (n-i-1) + n -2 = i \geqslant 0. 
\]
It follows that 
\[ \Res_\infty\frac{1}{\theta -t }  \D_{\f}
(t^i)\eta_\f^*   =0 .
\]
On the other hand, 
\begin{align*}
\Res_{\theta}\frac{1}{\theta -t }  \D_{\f} (t^i) \eta_{\f}^*=    \Res_{\theta} \frac{1}{t-\theta  }  \D_{\f}(t^i) \frac{1}{\f} dt   =  \frac{ \D_{\f}(\theta^i) }{\f(\theta)} .
\end{align*}
Therefore, $ C_i =   \frac{ \D_{\f}(\theta^i) }{\f(\theta)}$. Substituting this into \eqref{eq:Ci}  implies the first equality. 

We notice that 
\[ 
 \dfrac{1}{\theta-t} -  \frac{1}{\f (\theta)} \frac{ \f (\theta) - \f(t) }{\theta - t }  =  \frac{\f(t)}{(\theta -t) \f (\theta)}  
\]is divisible by $ \f $. So 
\begin{equation}\label{eq:ff}
     [\dfrac{1}{\theta-t}]_\f \equiv  \frac{1}{\f (\theta)} \frac{ \f (\theta) - \f(t) }{\theta - t } \Mod \f(t).
\end{equation} 
The second equality is verified by comparing the $t$‑degrees in \eqref{eq:ff}.
\end{proof}
The uniqueness of the $\f$-remainder, together with Lemma \ref{lem:fremainder}, gives
\[  \sum_{i=0}^{n-1} \frac{\D_{\f}(\theta^{i})}{\f(\theta)} t^{i} = \frac{1}{\f(\theta)} \frac{\f(\theta) - \f(t)}{\theta - t}.
\]
Hence,
\[ \oper_{\f}^{(2)}(X_1, X_2) = \frac{\f(X_1) - \f(X_2)}{X_1 - X_2},
\]
which agrees with Lemma \ref{lem:X1X2equal} (see also \eqref{eq:operf}).

\subsection{Evaluation of Hasse-Schmidt derivatives}
To extend Lemma \ref{lem:Lagrange} to $\p^k$-remainder, we need to take Hasse-Schmidt derivatives. 
\begin{defn}{\label{def:derivative}}
For $\omega(t)\in\mathbb{T}$, we define $ \delta_l^t \omega  $ to be the $l$-th Hasse-Schmidt derivative of $\omega(t) $, that is the $l$-th coefficient in the expression 
\[
\omega(t+X) = \sum_{l=0}^{\infty} \delta_l^t (\omega) X^l.
\]
\end{defn}
In particular, the zeroth Hasse-Schmidt derivative is trivial, i.e., $ \delta_0^t \omega = \omega $; and the first Hasse-Schmidt derivative is the classical derivative, i.e., $\delta_1^t\omega(t)=\dfrac{d\omega}{dt}$.

\begin{example}
We compute the Hasse-Schmidt derivatives for  the function $ \omega(t):=(\theta - t)^{-1} $ with parameter $ \theta $. From Taylor expansion, we obtain  
    \[\dfrac{1}{\theta- (t+X) }=\dfrac{1}{\theta-t}\left(\dfrac{1}{1-\frac{X}{\theta-t}}\right)=\dfrac{1}{\theta-t}\left(1+\dfrac{X}{\theta-t}+\left(\dfrac{X}{\theta-t}\right)^2+\cdots\right).
\]
So we have 
\begin{equation}\label{eq:derivativeThetaT}
\delta^t_l \left(\dfrac{1}{\theta-t}\right)=\dfrac{1}{(\theta-t)^{l+1}}.
\end{equation}
\end{example}
% \begin{example}
% % Suppose $a\in \mathbb{C}_\infty$, $k,l\geqslant 0$. 
% The Hasse-Schmidt derivatives for the monomial $(X-a)^k$ is given by 
% \[\delta_l^t  (X-a)^k  =\binom{k}{l}(X-a)^{k-l},\]
% where $ l ,k \geqslant 0 $ and 
% \[\binom{k}{l}=\begin{cases}
%     k(k-1)\cdots (k-l+1)\quad & l\leqslant k\\
%     0\quad &l>k
% \end{cases}.\]
% \end{example}

\begin{lem}{\label{lem:product-deriviate}}
    For $\omega_1,\omega_2,\dots,\omega_r\in\mathbb{T}$, we have
    \[\delta_l^t(\prod_{i=1}^r \omega_i)=\sum_{l_1+\cdots+l_r=l}\delta^t_{l_1}(\omega_1)\cdots\delta^t_{l_r}(\omega_r).\]
\end{lem}
As a consequence, we have the following result.
\begin{lem}\label{lem:delta_kp^k}
Let $ \p $ be an irreducible polynomial. 
\begin{enumerate}
    \item For $ l< k $, $ \delta_l^t \p^k \equiv 0  \Mod \p(t) $; 
    \item For $ l = k $, $\delta^t_l \p^k$ is coprime to $ \p. $ 
\end{enumerate}  
\end{lem}

As the following lemma demonstrates, the vanishing of the first $k$ Hasse-Schmidt derivatives at a point implies that the polynomial has a zero of order at least $k$ there.
\begin{lem}{\label{lem:zero}}
Let $f\in \mathbb{C}_\infty[t] $ be a polynomial. Assume that $ \zeta \in \mathbb{C}_\infty$ is a root of $ \delta^t_l(f)(\zeta)=0$ for all $l<k$. Then $(t-\zeta)^k|f$.
\end{lem}
\begin{prop}
    Let $ \p $ be an irreducible polynomial of degree $d$.
    Let $ \zeta_1,\cdots, \zeta_d $ be the roots of $ \p $.
    The $\p^k $-remainder of $ \omega  $ is the unique polynomial $ \lambda \in \mathbb{C}_{\infty}[t] $ of degree $\leqslant d k -1$ such that 
    \begin{equation}{\label{eq:polynimial-equal}}
        \delta_{l}^{t}\omega(\zeta_j ) =  \delta_{l}^{t}\lambda (\zeta_j) 
       \end{equation}
    for $ l = 0, \cdots, k-1$, $j= 1, \cdots, d $.
\end{prop}

\begin{proof}
Firstly, we show that $ \lambda =[\omega]_{\p^k} $ satisfies the condition \eqref{eq:polynimial-equal}.
By the definition of $\p^k$- remainder, we have  $\omega-[\omega]_{\p^k}\in \ker \ev_{\p^k}=\p^k\mathbb{T} $, i.e., $\omega-[\omega]_{\p^k}=\p^k \gamma $ for some $ \gamma \in \T $. From Lemmas \ref{lem:product-deriviate} and \ref{lem:delta_kp^k}, for all $l<k,1\leqslant j\leqslant d$, we have
\[\delta_l^t(\omega-[\omega]_{\p^k})(\zeta_j)=\delta_l^t(\p^k\gamma)(\zeta_j)=\sum_{l_1+l_2=l}\delta_{l_1}^t\p^k(\zeta_j) \cdot \delta_{l_2}^t\gamma(\zeta_j)=0.
\]
Next, we check the uniqueness.
Suppose that both $\lambda_1$ and $\lambda_2$ satisfy \eqref{eq:polynimial-equal}.
It follows from Lemma \ref{lem:zero} that $ \lambda_1-\lambda_2$ is divisible by $\p^k$. Since both degrees are less than $kd-1$, they are indeed identical.
\end{proof}

\section{The $\f$-Remainder of Anderson Generating Functions}\label{sec:generating}
In this section, we investigate the $\f$-remainder of $\omega_z $ for a fixed Drinfeld module $ \phi$. 
\subsection{Anderson Generating Functions}
 % Let us recall the definition of Anderson generating functions.  
 Assume that the rank-$r$ Drinfeld module is represented as 
 \begin{equation}\label{eq:Drinfeld}
 \phi_\xx = \sum_{i=0}^r  g_i \tau^i,  
 \end{equation}
 where   $ g_0 = \theta$, $ g_1,\cdots, g_{r-1} \in K, g_r \in K^\times $.
Let $ \exp_\phi $ be the exponential map of $ \phi $, i.e., an $\mathbb{F}_q$-linear function satisfying 
\[
    \phi_{ a(\xx)} \exp_\phi (\mu) = \exp_{\phi} (a(\theta) \mu) 
\]
for any $a(\xx) \in \xA $.
 It is known that $\exp_\phi$ can be written as 
 \[
 \exp_{\phi} (\mu) = \sum_{i = 0 }^\infty \frac{\mu^{q^i}}{D_i} ,
 \]
 for some nonzero coefficients $D_i \in \mathbb{C}_{\infty}$.

 The kernel $ \Lambda_\phi $ of $ \exp_{\phi} $ is called the lattice of $ \phi $. Due to the Drinfeld-Riemann uniformization theorem,  the lattice $ \Lambda_ \phi $ of $ \phi $ is in one-to-one correspondence with $ \phi $ up to isomorphisms.
 \begin{defn}
  The generating functions of $\phi$ associated with $ z $ are given by the series
\[
 \omega_z(t) = \sum_{k = 0 }^{\infty }\exp_{\phi} (\frac{z }{\theta^{k+1}})t^k .
\] 
 \end{defn}
It is trivial to check that  $ \omega_z \in \T $.
Pellarin's identity below is important for our main results in this section.
\begin{prop}[ \cite{MR2487735}*{Section 4.2} ]
   For $ z \in \mathbb{C}_{\infty}$, we have an identity in $\T$,
   \begin{equation}\label{eq:Pellarin}
     \omega_z(t)  = \sum_{i=0}^\infty \frac{ z^{q^i}
}{D_i ( \theta^{q^i} - t ) }.
\end{equation}
    Furthermore, $\omega_z(t)$ extends to a meromorphic function on $ \mathbb{C}_\infty$ with simple poles at $ t = \theta^{q^j}$ where $ j \geqslant 0  $. 
\end{prop}

% Then 
% \[
%     t G(U;t) = G(\theta U; t ) + \exp_{\phi}(U). 
% \]
% \[
%     G(U;t)^{(1)} = (t - \theta) G(U; t ) + \exp_{\phi}(U). 
% \]
Let $ \omega^{(k)}(t)$ denote the $k$-th Frobenius twist on the coefficients of $\omega(t) \in \T $. 
For each $ z \in \Lambda_\phi $, the generating function $ \omega_z $ is a solution of the Frobenius differential equation: 
\begin{equation}\label{eq:Frobenius}
   t \cdot \omega  =  \phi_\xx  \omega  := g_r \omega^{(r)} + g_{r-1} \omega^{(r-1)} + \cdots + g_1\omega^{(1)} + \theta \omega.
\end{equation}

\subsection{Weil operator and remainder}
% In this section, we wish to 
In this section, we wish to determine the $\f$-remainder of $\omega_z$. Since $\mathbb{F}_q[\theta] $ is naturally a subring of $ \mathbb{C}_{\infty} $, we can view $\mathbb{C}_{\infty} $ as an $ \mathbb{F}_q[\theta] $-module via ordinary multiplication.  Moreover, the Drinfeld module $ \phi : \xA \cong \mathbb{F}_q[\theta] \to \mathbb{C}_{\infty} \{ \tau \} $ 
endows $ \mathbb{C}_{\infty} $ with another $ \mathbb{F}_q[\theta] $-module structure. To distinguish these two structures, we introduce the following notation for the Drinfeld action.
\begin{notation}[Drinfeld Action]\label{no:action}
For a polynomial $a(\xx) \in \xA $, we define its action on $\mathbb{C}_{\infty}$ via $\phi$ by  
\[
a(\xx) \diamond_{\phi} (\mu) = \phi_{a(\xx)}(\mu),
\]  
where $\phi_{a(\xx)}$ denotes the image of $a(\xx)$ under $ \phi $. For convenience, we extend $\diamond_{\phi}$ to polynomials $a(\xx) \in \mathbb{F}_{q}(t)[\xx]$ acting $ \mathbb{C}_\infty(t) $ by letting elements of $\mathbb{F}_{q}(t)$ act as scalars.
\end{notation}
\begin{prop}\label{prop:action}
For $ a(t) \in \mathbb{F}_q[t]$, and $ \eta^* \in \f^{-1} \Omega $, we have
 \[
    a(\xx) \diamond_{\phi} (\langle \omega_z , \eta^*  \rangle_{\T})  =  \langle a(t)    \omega_z , \eta^*   \rangle_{\T} =  \langle   \omega_z , a(t)  \eta^*   \rangle_{\T}.
\]
\end{prop}
\begin{proof}
    It suffices to show that 
 \[
    \xx  \diamond_{\phi}( \langle \omega_z , \eta^*   \rangle_{\T})  =  \langle t    \omega_z , \eta^*   \rangle_{\T}  . 
\]
It is evident that 
\[
\langle   \omega_z , \eta^* \rangle_{\T}^{q^i} = \langle   \omega_z^{(i)} , \eta^*  \rangle_{\T} 
\]
for $i \geqslant 0$. 
Assume that $ \phi $ is of the form \eqref{eq:Drinfeld}. 
Then the Frobenius differential equation \eqref{eq:Frobenius} yields
\begin{align*}
\xx  \diamond_{\phi} (\langle \omega_z , \eta^*  \rangle_{\T}) & = \phi_\xx (\langle \omega_z , \eta^*  \rangle_{\T}) \\
& = \theta \langle   \omega_z  , \eta^*  \rangle_{\T} + g_1 \langle   \omega_z  , \eta^*   \rangle_{\T}^q + \cdots + g_r \langle   \omega_z  , \eta^*  \rangle_{\T}^{q^r} \\
& = \langle   \theta \omega_z , \eta^*   \rangle_{\T} +  \langle   g_1\omega_z^{(1)} , \eta^*   \rangle_{\T} + \cdots +  \langle   g_r\omega_z^{(r)} , \eta^* \rangle_{\T} \\
& = \langle t \cdot \omega_z ,  \eta^*  \rangle_{\T}.
\end{align*}
 This implies the first equality. The second one has been shown in (3) of Lemma \ref{lem:pro}.
 \end{proof}

\begin{thm}\label{thm:expansion}
Let $ \oper_{\f}^{(2)}(\xx,t) $ be the rank-two Weil operator. Then $\f $-remainder of $\omega_z$ is given by 
    \begin{align*}
    [\omega_z(t)]_\f & = \oper_{\f}^{(2)}(\xx, t)\diamond_{\phi} (C_{z,n -1})  \\
    & = \sum_{i=0}^{n-1} \Big( \D_{\f}(\xx^i ) \diamond_{\phi} (C_{z,n -1} )\Big)  t^i  .
\end{align*}
\end{thm}
\begin{proof}
We express the $\f$-remainder of  $ \omega_z$ as
\[
\omega_z = \sum_{i=0}^{n -1} C_{z,i} t^i  .
\]
 Then  
\begin{align*} 
C_{z,i} &= \langle  \omega_z,    \D_{\f}(t^i)   \eta^*_{\f} \rangle_{\T} \quad \text{By Proposition \ref{prop:coefficients} } \\
 & = \D_{\f}(\xx^i) \diamond_{\phi}( \langle  \omega_z,     \eta^*_{\f} \rangle_{\T} )\quad \text{By Proposition \ref{prop:action} } \\
 & = \D_{\f}(\xx^i) \diamond_{\phi} (C_{z, n-1}) \quad \text{By Proposition \ref{prop:coefficients}.}
\end{align*}
The theorem follows from the expression of $ \oper_{\f}^{(2)}(x, t) $. 
\end{proof}
  
\subsection{Residue Formula}
We have seen in Lemma \ref{lem:residue} that when $ \omega $ is a rational function, the pairing $ \langle \omega, \eta^* \rangle_{\T} $ can be written as the sum of residues of $ \omega \eta^* $ at infinity $\infty$ and all poles. It is natural to ask whether the formula still holds for general $ \omega \in \T $. The main challenge is that the residue at $\infty$ is not well-defined for general $ \omega \in \T $. However, for Anderson generating functions, the residue at infinity can be defined as the limit of residues of rational functions by applying  Pellarin's identity \eqref{eq:Pellarin}. One should notice that $ \omega_z $ has simple poles at $ \theta^{q^i} $ for $ i \geqslant 0 $.  

 \begin{lem}
For meromorphic differential $\eta^* \in \f^{-1} \Omega $, we have 
\[ \langle \omega_z , \eta^* \rangle_{\T}  =  \sum_{i=0}^\infty \Res_{\theta^{q^i}} \frac{z^{q^i}}{ D_i (\theta^{q^i}-t)} \eta^*=  \sum_{i=0}^\infty \Res_{\theta^{q^i}} \omega_z\eta^*. 
\] 
\end{lem}

\begin{proof}
 From  Pellarin's identity  \eqref{eq:Pellarin}, 
 \[
     \langle \omega_z , \eta^* \rangle_{\T}
= \langle \sum_{i=0}^\infty \frac{ z^{q^i}
}{D_i ( \theta^{q^i} - t ) }, \eta^* \rangle_{\T}  = \sum_{i=0}^\infty \langle \frac{ z^{q^i}
}{D_i ( \theta^{q^i} - t ) }, \eta^* \rangle_{\T} . 
\]
By the formula in Lemma \ref{lem:residue}, we have
\[
 \langle \frac{ z^{q^i}
}{D_i ( \theta^{q^i} - t ) }, \eta^* \rangle_{\T} 
 = \Res_{\infty} \frac{ z^{q^i}}{D_i ( \theta^{q^i} - t ) } \eta^* + \Res_{\theta^{q^i}} \frac{ z^{q^i}}{D_i ( \theta^{q^i} - t ) } \eta^* .
\]
It is evident that the residue of $ \frac{ z^{q^i}}{D_i ( \theta^{q^i} - t ) } \eta^*  $ at $ \infty $ is zero.
Thus, we obtain
\[ \langle \omega_z , \eta^* \rangle_{\T}  =  \sum_{i=0}^\infty \Res_{\theta^{q^i}} \frac{z^{q^i}}{ D_i (\theta^{q^i}-t)} \eta^* .
\]
Applying  Pellarin's identity  again, we obtain the second equality.

\end{proof}
In particular, for $ \eta^* = \eta^*_{\f} $, we have
\begin{align*}
     \langle \omega_z , \eta^*_{\f} \rangle_{\T}  & =  \sum_{i=0}^\infty \Res_{\theta^{q^i}} \frac{z^{q^i}}{ D_i (\theta^{q^i}-t)} \eta^*_{\f} \\
     & = \sum_{i=0}^\infty \frac{1}{ D_i } \left( \frac{z}{ \f(\theta)} \right)^{q^i} = \exp_{\phi}(\frac{z}{\f(\theta)}).
\end{align*}
So Proposition \ref{prop:coefficients} yields that 
 \begin{equation}\label{eq:leading}
       C_{z,n-1}    =  \langle \omega_z , \eta^*_{\f}  \rangle_{\T}  =  \exp_{\phi}( \frac{z}{\f (\theta) }). 
 \end{equation} 
Together with Theorem \ref{thm:expansion}, we obtain directly the following result.
\begin{cor}\label{cor:expansion}
     With notations in Theorem \ref{thm:expansion}, we obtain 
    \[ C_{z,i }=  \D_{\f}(\xx^i) \diamond_{\phi} \left(\exp_{\phi}(\frac{z}{\f(\theta) })\right) = \phi_{\D_{\f}(\xx^i) } \left(\exp_{\phi}(\frac{z}{\f(\theta) })\right) ,
    \]
and 
    \begin{equation}\label{eq:omegaz2}
    [\omega_z(t)]_\f = \oper^{(2)}_{\f}(\xx, t)\diamond_{\phi}  \exp_{\phi}(\frac{z}{\f(\theta)}) = \sum_{i=0}^{n-1} \Big( \D_{\f}(\xx^i ) \diamond_{\phi} \exp_{\phi}(\frac{z}{\f(\theta) }) \Big)  t^i .
\end{equation}
 
\end{cor}

\begin{cor}\label{cor:basis}
    Let $C_{z_i, j} $ be the coefficients in the $ \f$-remainder of $ \omega_{z_i} $.  
    \begin{enumerate}
        \item Then the torsion space $  \phi[\f]   $ is spanned by $ C_{z_i, j}  $ where 
    $i = 1,\cdots, r$, and $ j =0 ,\cdots , n -1 $ as a $\mathbb{F}_q $-vector space. 
    \item The function field
    $ \mathbb{F}_{q}(t)(\phi[\f] )$ equals 
    \[ \mathbb{F}_{q}(t)(\phi[\f] ) = \mathbb{F}_{q}(t)(C_{z_i, j}|{j < n }) =   \mathbb{F}_{q}(t)(C_{z_i, n -1}). \]
    \end{enumerate}
\end{cor}
\begin{proof}
    Notice that $  \phi[\f] $ is generated by $ \exp(\frac{z_i}{\f(\theta)}) $ as an $ \xA $-algebra. From Corollary \ref{cor:expansion},
    \[
      C_{z_i, j}     = \phi_{\D_{\f}(\xx^j)} (\exp(\frac{z_i}{\f(\theta)})).  
    \]
   Since $\D_{\f}(\xx^j)$ with $j =0 ,\cdots, n-1$ form a basis of $ \xA / \f(x) \xA  $ as an $\mathbb{F}_q$-vector space, we conclude the first assertion.

    The second part of the corollary follows from the first part and the fact that $ C_{z_i, j} $ is generated by $  C_{z_i, n-1 } =\exp(\frac{z_i}{\f(\theta)} ) $.
\end{proof}

\begin{cor}\label{cor:mu}
For $\mu \in \phi[\f ] $, there exists a unique $ z  $ in the lattice  $ \Lambda_\phi $ such that 
\[
    \mu = C_{z, n - 1 }. 
\]
\end{cor}
\begin{proof}
   By Corollary \ref{cor:basis}, we may assume that 
    \[ 
    \mu = \sum_{i=0}^{n-1} a_i(\xx) \diamond_{\phi}(\exp_{\phi}(\frac{z_i}{\f(\theta)}) 
    )\]
    Choose an  element $ z $ written as 
    \[
    z = \sum_{i=0}^{n-1} a_i(\theta) z_i  . 
    \]
    As in \eqref{eq:leading}, the leading coefficient of $ [\omega_z]_\f $ is given by 
    \[
    C_{z, n-1} = \exp_{\phi}(\frac{z}{\f(\theta)})= \exp_{\phi}(\frac{ \sum_{i=0}^{n-1} a_i(\theta) z_i}{\f(\theta)})  = \sum_{i=0}^{n-1} a_i(\xx) \diamond_{\phi} (\exp_{\phi}(\frac{z_i}{\f(\theta)})) = \mu.
    \]
    The uniqueness of $ z $ follows from the injectivity of $ \exp_{\phi} $.
\end{proof}

\subsection{Relation to Maurischat-Perkins' Results}
Let $ \p $ be a monic irreducible polynomial of degree $d$. As an application of Corollary \ref{cor:expansion}, this part aims to determine the $\p $-remainder of the $l$-th Hasse-Schmidt derivative  of $ \omega_z $. 
\begin{lem}\label{lem:lHS}
    For $l < k $, the $l$-th Hasse-Schmidt derivative of $ \oper_{\p^k}^{(2)}(\xx, t) $ (with respect to $t$) satisfies
    \begin{equation}\label{eq:pthetaOper}
           \delta_l^t \oper_{\p^k}^{(2)}(\xx, t)  \equiv  \p(\theta)^{k-l-1} \oper_{\p}^{(2)}(\xx, t)^{l+1}  \Mod \p(t). 
    \end{equation}
\end{lem}
\begin{proof}
It suffices to show that the equality
\begin{equation}\label{eq:primeOper}
      \delta_l^t \oper_{\p^k}^{(2)}(\xx, \zeta )   = \p(\xx)^{k-l-1} \oper_{\p}^{(2)}(\xx, \zeta)^{l+1} 
  \end{equation}
holds for any root $\zeta$ of $\p$.
Substituting $\oper^{(2)}_\p(\xx,t)$ by $\dfrac{\p(\xx)-\p(t)}{\xx-t}$ (see the formula \eqref{eq:operf}) into the right-hand side of \eqref{eq:primeOper} yields
\begin{equation}\label{eq:primetheta}
    \p(\xx)^{k-1-l}\oper^{(2)}_\p(\xx,\zeta)^{l+1}=\p(\xx)^{k-1-l}\left(\dfrac{\p(\xx)}{\xx-t}\right)^{l+1}=\dfrac{\p(\xx)^k}{(\xx-t)^{l+1}}.
\end{equation}
 
Applying the formula \eqref{eq:operf} again, the left-hand side of \eqref{eq:primeOper} equals 
\begin{align}
    \delta_l^t \oper^{(2)}_{\p^k}(\xx,\zeta) & =\delta_l^t\left(\dfrac{\p^k(\xx)-\p^k(t)}{\xx-t}\right) \nonumber \\
    & =\sum_{ l_1+l_2=l}\delta_{ l_1}^t\left(\p^k(\xx)-\p^k(t)\right)\mid_{t=\zeta} \cdot  \delta_{ l_2}^t\left(\dfrac{1}{\xx-t}\right) \mid_{t=\zeta} \label{eq:deltalt}.
\end{align}
By Lemma \ref{lem:delta_kp^k}, when $ l_1 <k$, we have $\delta_{l_1}^t\p^k(t) \mid _{t = \zeta}=0$. It follows that $\delta_{l_1}^t(\p^k(\xx)-\p^k(t))|_{t=\zeta} $ vanishes for $l_1>0$. We can then use the formula \eqref{eq:derivativeThetaT} to get
\begin{equation}\label{eq:deltal}
\delta_l^t\oper_{\p^k}^{(2)}(\xx,\zeta)=\delta_{ l}^t\left(\dfrac{1}{\xx-t}\right) \mid_{t=\zeta} =\dfrac{\p^k(\xx)}{(\xx-\zeta)^{l+1}}.
\end{equation}
Comparing \eqref{eq:primetheta} with \eqref{eq:deltal}, we get the equality \eqref{eq:primeOper}. 
\end{proof}
In particular, when $l = 0$, we have 
\[
      \oper^{(2)}_{\p^k}(\xx, t) \equiv   \p^{k -1}(\xx) \cdot \oper^{(2)}_{\p }(\xx, t)   \Mod\p(t).
    \]
This coincides with \eqref{eq:pk}.

\begin{notation}
     For $l \geqslant 0 $, we choose coefficients $E_i^{(l)}(\xx)$  such that  
    \[
      \sum_{i=0}^{d-1}  E_i^{(l)}(\xx) t^i \equiv    \oper_{\p}^{(2)}(\xx, t)^{l+1} \Mod \p(t) .   
    \]
\end{notation}
Note that the degree of $ E_i^l(\xx) $ is less than $ (l+1 )d$.
In particular, for $ l = 0 $, we know that $ \oper_{\p}^{(2)}(\xx, t) = \sum_{i=0}^{d-1} \D_\p(\xx^i) t^i $ is contained in $ \mathbb{F}_q[\xx, t]_{< d} $.  
It follows that $ E_i^{(0)}(\xx) = \D_\p(\xx^i)$. 

We recover \cite{MaurischatPerkins2022}*{Proposition 3.4} as follows.
\begin{cor}\label{cor:MP}
   With the notations above, we have the congruence
%     \[
%     \delta_l^{t} \oper_{\p^k}(\xx, t) \equiv \sum_{i=0}^{n-1}  ( \p^{k-l-1}(\xx) \cdot E_i^{(l)}(\xx) )  t^i   \Mod \p(t).
%     \]
%   Therefore, 
    \[
        \delta_l^t \omega_z \equiv \sum_{i=0}^{d-1}   \exp_{\phi}(\frac{ E_i^{(l)}(\theta)  z}{\p^{l+1}(\theta)}) t^i \Mod \p(t).
    \]
\end{cor}
\begin{proof}
From Equation \eqref{eq:omegaz2}, the $\p^k$-remainder of $ \omega_z $ is given by $ \oper_{\p^k}^{(2)}(\xx,t) \diamond_{\phi} (\exp_{\phi}(\frac{z}{\p^k(\theta)}))$. By Lemma \ref{lem:delta_kp^k}, when $0\leqslant l < k$, the $\p$-remainder of  $\delta_l^t \omega_z$ is congruent to  
    \[
         \delta_l^{t} \left( \oper_{\p^k}^{(2)}(\xx,t) \diamond_{\phi} (\exp_{\phi}(\frac{z}{\p^k(\theta)})) \right)  =  \delta_l^{t} \left( \oper_{\p^k}^{(2)}(\xx,t) \right) \diamond_{\phi}  \exp_{\phi}(\frac{z}{\p^k(\theta)})   
    \]
modulo $ \p(t) $. 
Applying Lemma \ref{lem:lHS}, we obtain 
\[ \delta_l^{t} \left( \oper_{\p^k}^{(2)}(\xx,t) \right) \equiv \p(\xx)^{k-l-1}\oper_{\p}^{(2)}(\xx, t)^{l+1}  \equiv \sum_{i=0}^{d-1} \p(\xx)^{k-l-1} E_i^{(l)}(\xx) t^i \Mod \p(t) .   
\]
Thus, we see that 
     \begin{align*}
        \delta_l^t \omega_z  & \equiv \sum_{i=0}^{d-1} \left( \p(\xx)^{k-l-1} E_i^{(l)}(\xx) t^i  \right) \diamond_{\phi} (\exp_{\phi}(\frac{z}{\p^k(\theta)}) ) \Mod \p(t)  \\
        & 
         \equiv\sum_{i=0}^{d-1}      \exp_{\phi}(\frac{E_i^{(l)}(\theta) z}{\p^{l+1}(\theta)})t^i  \Mod \p(t)  . 
     \end{align*}

\end{proof}

\section{Weil Pairing}\label{sec:weil} 
In this section, we state  our theorem concerning the Weil pairing. 
\subsection{Moore Determinant of Generating Functions}
To begin with, we recall the Moore determinant of generating functions, and explain why the corresponding rank-one Drinfeld module $ \psi $ is given by 
\begin{equation}\label{eq:psi}
\psi_{\xx} =  (-1)^{r-1} g_r \tau + \theta . 
\end{equation}
Due to Definition 4.1 in \cite{HY93}, we define the Moore determinant $ \M: \T^{\otimes r} \to \T $ as follows.
\begin{defn}\label{defn:Moore}
For $\omega_{ 1}, \cdots, \omega_{ r} \in \T $, 
$ \M(\omega_{ 1}, \cdots, \omega_{ r})$ is the determinant of 
\[
\begin{pmatrix}
\omega_1(t) & \omega_2(t) & \cdots & \omega_r(t) \\
\omega_1^{(1)}(t) & \omega_2^{(1)}(t) & \cdots & \omega_r^{(1)}(t) \\
\vdots & \vdots & \ddots & \vdots \\
\omega_1^{(r-1)}(t) & \omega_2^{(r-1)}(t) & \cdots & \omega_r^{(r-1)}(t) \\
\end{pmatrix}
\]
\end{defn}
Notice that the function $ \kappa(t):= \M( \omega_{ 1}, \cdots, \omega_{ r}) $ does not vanish if and only if $\omega_i$'s are $\mathbb{F}_q$-linearly independent.
We restate Lemma 4.4 in \cite{HY93} as follows. 
\begin{prop}\label{prop:kappa}
Suppose that $\omega_i = \omega_{z_i}$ are the Anderson generating functions associated with a basis $z_1, \cdots, z_r $ of $ \Lambda_{\phi}$. Then the function $ \kappa(t) $ satisfies
\[ (-1)^{r-1} g_r  \kappa^{(1)} = ( t - \theta) \kappa. \]
\end{prop}

\begin{proof} 
The proof is straightforward:
\begin{align*}
   g_r \kappa^{(1)}&=\det\begin{pmatrix}
\omega^{(1)}_1(t) & \omega^{(1)}_2(t) & \cdots & \omega^{(1)}_r(t) \\
\omega_1^{(2)}(t) & \omega_2^{(2)}(t) & \cdots & \omega_r^{(2)}(t) \\
\vdots & \vdots & \ddots & \vdots \\
g_r\omega_1^{(r)}(t) & g_r\omega_2^{(r)}(t) & \cdots & g_r\omega_r^{(r)}(t) \\
\end{pmatrix}\\
&=\det\begin{pmatrix}
\omega^{(1)}_1(t) & \omega^{(1)}_2(t) & \cdots & \omega^{(1)}_r(t) \\
\omega_1^{(2)}(t) & \omega_2^{(2)}(t) & \cdots & \omega_r^{(2)}(t) \\
\vdots & \vdots & \ddots & \vdots \\
(t-\theta)\omega_1-\sum_{j=0}^{r-1} g_k\omega_1^{(k)}(t) & (t-\theta)\omega_2-\sum_{j=0}^{r-1} g_k\omega_2^{(k)}(t) & \cdots & (t-\theta)\omega_r-\sum_{j=0}^{r-1} g_k\omega_r^{(k)}(t) \\
\end{pmatrix}\\
&=\det\begin{pmatrix}
\omega^{(1)}_1(t) & \omega^{(1)}_2(t) & \cdots & \omega^{(1)}_r(t) \\
\omega_1^{(2)}(t) & \omega_2^{(2)}(t) & \cdots & \omega_r^{(2)}(t) \\
\vdots & \vdots & \ddots & \vdots \\
(t-\theta)\omega_1 & (t-\theta)\omega_2 & \cdots & (t-\theta)\omega_r \\
\end{pmatrix}\\
&=(-1)^{r-1}(t-\theta)\det\begin{pmatrix}
\omega_1(t) & \omega_2(t) & \cdots & \omega_r(t) \\
\omega_1^{(1)}(t) & \omega_2^{(1)}(t) & \cdots & \omega_r^{(1)}(t) \\
\vdots & \vdots & \ddots & \vdots \\
\omega_1^{(r-1)}(t) & \omega_2^{(r-1)}(t) & \cdots & \omega_r^{(r-1)}(t) \\
\end{pmatrix}\\
&=(-1)^{r-1}(t-\theta)\kappa.
\end{align*}
\end{proof}
 Proposition \ref{prop:kappa} states that 
$ \kappa(t) $ verifies the Frobenius differential equation
\[
    (-1)^{r-1} g_r \omega^{(1)} + \theta \omega = t \omega,
\]
which is a special form of \eqref{eq:Frobenius}. 
It follows that $ \kappa(t) $ is an Anderson generating function of $ \psi $. So the definition below is natural.
\begin{defn}
    The exterior product of $ \phi $ is a rank-one Drinfeld module of the form \eqref{eq:psi}. 
\end{defn}
\subsection{Katen's Formula for Weil Pairing}\label{sec:Katen}
This action $\diamond_{\phi} $ in Notation \ref{no:action} can be extended to the case of polynomials in $ r$-variables.
 \begin{defn}\label{defn:DrinfeldAction}
 Let $\mathbb{C}_{\infty}^{ \otimes r} = \mathbb{C}_{\infty}\otimes_{\mathbb{F}_q}\cdots \otimes_{\mathbb{F}_q} \mathbb{C}_{\infty} $.
 \begin{enumerate}
     \item  For $   \mu_1\otimes \cdots\otimes \mu_r  \in \mathbb{C}_{\infty}^{ \otimes r }$, we define the action by the monomial $ X_1^{a_1} \cdots X_n^{a_n} $ as 
 \[
X_1^{a_1} \cdots X_n^{a_n} \diamond_{\phi} \left( \mu_1\otimes \cdots\otimes \mu_r\right) =  \phi_{\xx^{a_1}}(\mu_1) \otimes \cdots\otimes \phi_{\xx^{a_n}}(\mu_n)  .
\]
\item We assume that $ g(t) \in \mathbb{F}_q(t) $ acts as a scalar on $ \mathbb{C}_{\infty}^{ \otimes r }$, i.e., 
\[ g(t) \diamond_{\phi} \left( \mu_1\otimes \cdots\otimes \mu_r\right)= g(t)\cdot\left( \mu_1\otimes \cdots\otimes \mu_r\right).
\]
\item For a polynomial $F \in \mathbb{F}_q[X_1,\cdots, X_r,t]$ of the form:
\[  F =   \sum_{a_1,\cdots,a_n} F_{a_1, \cdots, a_n }(t) X_1^{a_1} \cdots X_n^{a_n},  \]
we define the action
\[
    F \diamond_{\phi} \left( \mu_1\otimes \cdots\otimes \mu_r\right) = \sum_{a_1, \cdots, a_n} F_{a_1, \cdots, a_n}(t) \cdot  X_1^{a_1} \cdots X_n^{a_n} \diamond_{\phi} \left( \mu_1\otimes \cdots\otimes \mu_r\right).
\]
 \end{enumerate}
 \end{defn}

% \begin{lem}\label{lem:det}
%     For  a polynomial $ F = F(X_1,\cdots, X_n )$ and $ k \in \{ 1, \cdots, n \}$, we define 
%     \[
%        F \diamond_{\phi}^k ( \mu_1, \cdots, \mu_n ) = (F - F(\mathbf{0}))\diamond_{\phi}  ( \mu_1, \cdots, \mu_n )  + ( \mu_1, \cdots, F(\mathbf{0}) \mu_k ,\cdots, \mu_n ). 
%     \]
%   Denote by $ \M $ the Moore determinant. Then 
%     $
%       \M( f \diamond_{\phi}^k ( \mu_1, \cdots, \mu_n ) ) 
%     $is independent of $ k $. 
% \end{lem}
% \begin{proof}
% It suffice to prove the case when $ F $ equals a constant $c$. Direct computation yields 
% \begin{align*}
%    \mathcal{M}(\mu_1,\cdots,c\mu_k,\cdots,\mu_n)=\det\left(\begin{array}{cccc}
% \mu_1 & \mu_1^q &\cdots &\mu_1^{q^{n-1}}\\
% \vdots & \vdots &      & \vdots\\
% c\mu_k & (c\mu_k)^q &   \cdots & (c\mu_k)^{q^{n-1}}\\
% \vdots &\vdots &   & \vdots\\
% \mu_n & \mu_n^q &   \cdots & \mu_n^{q^{n-1}}    
% \end{array}\right)=c^n \mathcal{M}(\mu_1,\cdots,\mu_k,\cdots,\mu_n),
% \end{align*}
% which is independent of $k$.
% \end{proof}
The Moore determinant $\M$ on $ \mathbb{C}_{\infty}^{\otimes r} $ is the alternating, $\mathbb{F}_q$-linear map
\[ \M:  \mathbb{C}_{\infty}  ^{\otimes r} \to \mathbb{C}_{\infty}, \qquad (\mu_1,\cdots,\mu_r) \mapsto \det\bigl( \mu_i^{q^{j-1}} \bigr)_{1\leqslant i,j\leqslant r},
\]
as defined in Definition~\ref{defn:Moore} (omitting the auxiliary variable $t$). Let $ \oper_{\f}^{(r)} $ be the rank-$r$ Weil operator with modulus $\f$. 
With these notations, the Weil pairing 
 \[
 \Weil_\f : \phi[\f] \times \cdots \times\phi[\f] \to \psi[\f]
 \]
 is defined as 
 \begin{equation}\label{eq:weil}
 \Weil_\f (\mu_1, \cdots, \mu_r ) = \M (  \oper_{\f}^{(r)} \diamond_{\phi} ( \mu_1 \otimes \cdots \otimes \mu_r ) ).  
 \end{equation}
 
% Though the constant term of $ \oper_{\f}^{(r)} $ is nontrivial, Equation \eqref{eq:weil} makes sense due to Lemma \ref{lem:det}.

 \subsection{New Insight on the Weil Pairing}
Let $ \f $ be a polynomial of degree $ n $. Suppose that $ \phi $ is a rank-$r$ Drinfeld module. Let $ z_1, \cdots, z_r $ be elements in the lattice $ \Lambda_{\phi}$ of $ \phi $. From Theorem \ref{thm:expansion}, we see that the $\f $-remainder of $ \omega_{z_i}(t) $ is 
\begin{equation}\label{eq:zit}
[\omega_{z_i}(t)]_{\f } =  \sum_{j =0}^{n -1} (\D_{\f }(\xx^j) \diamond_{\phi} ( C_{z_i,n -1})) t^{j}
\end{equation}
for each $i = 1, \cdots, r$, where $ C_{z_i, n-1} $ is the $(n -1)$-th coefficient. Let $\kappa(t) $ be the Moore determinant of $ \omega_{z_1}, \cdots, \omega_{z_r} $.  Assume that the $ \f $-remainder of $ \kappa $ is given by 
   \[  [\kappa(t)]_{\f } = \sum_{j=0}^{n-1} K_j t^j  \]
   for some constants $ K_0,\cdots , K_{n-1} \in \mathbb{C}_{\infty}$.
% Let $ \psi $ be the wedge product of $ \phi $. 
\begin{thm}\label{thm:main}
    With the notations above, the following assertions hold. 
   \begin{enumerate}
   \item The $ \f $-remainder of $ \kappa $ is given by 
   \[ \sum_{i=0}^{n -1} K_i t^i  = \M \Big( \oper_\f ^{(r+1)}(X_1, \cdots, X_r,t ) \diamond_{\phi}  (C_{z_1, n-1} \otimes \cdots \otimes C_{z_r,n-1} ) \Big),
   \]
   where $ \M $ denotes the Moore determinant restricted to $\mathbb{C}_{\infty}[t]^{\otimes r}$ (see Definition \ref{defn:Moore}).
\item  The leading coefficient $ K_{n -1} $ is the Weil pairing of $ C_{z_j, n -1}$ with $ 1 \leqslant j \leqslant r$, i.e., 
    \[
    K_{n -1}= \Weil_\f (C_{z_1, n -1},\cdots, C_{z_r, n-1} )  .
   \]
   \end{enumerate}
  
\end{thm}

\begin{proof}
Let $ \mathbf{c} =   C_{z_1, n -1} \otimes \cdots \otimes C_{z_r, n -1}  $.
By the definition of $\kappa(t)$ and Equation \eqref{eq:zit}, we have
\begin{align*}
[\kappa(t)]_\f
    &\equiv \M \left(  \sum_{j_1}\phi_{\D_{\f }(\xx^{j_1})}(C_{z_1,n -1}) t^{j_1},\cdots, \sum_{j_r}\phi_{\D_{\f }(\xx^{j_r})}(C_{z_r,n -1}) t^{j_r} \right)\Mod\f(t)  \\
    & \equiv \M \left( \Big( \sum_{j_1} \D_{\f }(X_1^{j_1}) t^{j_1} \sum_{j_2} \D_{\f }(X_2^{j_2})   t^{j_2}  \cdots  \sum_{j_r} \D_{\f }(X_{r}^{j_r})   t^{j_r}  \Big) \diamond_{\phi}  \mathbf{c} \right)\Mod\f(t) \\
     & \equiv\M \left( \left( \oper_{\f }^{(2)}(X_1,t)\oper_{\f }^{(2)}(X_2,t)\cdots \oper_{\f }^{(2)}(X_r,t)  \right) \diamond_{\phi}\mathbf{c} \right)\Mod\f(t)\\
      & \equiv \M \left( \oper_{\f}^{(r+1)}(X_1,\cdots, X_r , t) \diamond_{\phi}\mathbf{c} \right)\Mod\f(t).
\end{align*}
Since each term of $\oper_{\f}^{(r+1)}(X_1,\cdots, X_r , t)$ has $t$-degree $< n$, we obtain the first assertion. 

Applying Corollary \ref{cor:recursive}, the coefficient of $ t^{ n -1}$-term in  $\oper_{\f}^{(r+1)} (X_1,\cdots, X_r, t) $ is given by rank-$r$ Weil operator
$ \oper_{\f}^{(r)} (X_1,\cdots, X_r) $. From the first assertion, the coefficient $ K_{n -1} $ in the $\f  $-remainder of $ \kappa(t) $ is given by 
 \[ \M( \oper_{\f }^{(r)} (X_1,\cdots, X_r) \diamond_{\phi} \mathbf{c}  )  . 
\]
This expression coincides with $ \Weil_\f (C_{z_1,n -1}, \cdots, C_{z_r,n -1} )$ in \eqref{eq:weil}. 
\end{proof}
\subsection{Summary}\label{sec:summary}
In conclusion, we are able to give a new interpretation for Weil pairings with modulus $ \f  $.
Given an $r$-tuple $(\mu_1,\cdots,\mu_r)$ with $\mu_i \in \phi[\f]$, we associate to each $\mu_i$ an element $z_i$ in the lattice of $\phi$ such that  
\[
\mu_i = \exp_\phi\!\left(\frac{z_i}{\f(\theta)}\right),\qquad i=1,\dots,r.
\]  
Equivalently, by Corollary \ref{cor:expansion}, $\mu_i$ is equal to the $(n-1)$-st coefficient of the $\f$-remainder of the generating function $\omega_{z_i}(t)$.  
Let $\kappa(t)$ denote the Moore determinant of $\omega_{z_1}(t),\cdots,\omega_{z_r}(t)$.  
Then the value of the Weil pairing for the $r$-tuple $(\mu_1,\cdots,\mu_r)$ coincides with the $(n-1)$-st coefficient of the $\f$-remainder of $\kappa(t)$.
   
 \bibliographystyle{amsplain}
 \bibliography{papers}

\end{document}